\documentclass[11pt,a4paper]{article}
\usepackage[utf8]{inputenc}

\input{macros}

\begin{document}

\title{Lyapunov stabilization for nonlocal traffic flow models}
\author{Jan Friedrich\footnotemark[1], \; Simone Göttlich\footnotemark[2], \; Michael Herty\footnotemark[1]}
\footnotetext[1]{RWTH Aachen University, Institute of Applied Mathematics, 52064 Aachen, Germany (\{friedrich,herty\}@igpm.rwth-aachen.de).}
\footnotetext[2]{University of Mannheim, Department of Mathematics, 68131 Mannheim, Germany (goettlich@uni-mannheim.de).}
\date{\today}

\newcommand{\mh}[1]{ {\color{red}#1}}
\newcommand{\rv}[1]{{#1}}

\maketitle

\begin{abstract}
Using a nonlocal second-order traffic flow model we present an approach to control the dynamics towards a steady state. \rv{The system is controlled by the leading vehicle driving at a prescribed velocity and also determines the steady state}. Thereby, we consider both, the microscopic and macroscopic scales. We show that the fixed point of the microscopic traffic flow model is asymptotically stable \rv{for any kernel function}. Then, we present Lyapunov functions for both, the microscopic and macroscopic scale, and compute the explicit rates at which the \rv{vehicles influenced by the nonlocal term tend} towards the stationary solution. \rv{We obtain the stabilization effect for a constant kernel function and arbitrary initial data or concave kernels and monotone initial data.} Numerical examples demonstrate the theoretical results.
\end{abstract}

% REQUIRED
%\begin{keywords}

% REQUIRED
%\begin{MSCcodes}
%
%\end{MSCcodes}

\section{Introduction}
Progress in autonomous driving brings new challenges for the modelling of traffic flow.
To deal with those, approaches such as the classical Lighthill-Whitham-Richards (LWR) model \cite{lighthill1955kinematic, richards1956shockwaves} have been extended to include more information on the surrounding traffic, see for example~\cite{BlandinGoatin2016,friedrich2018godunov,GoatinScialanga2016,KeimerPflug2017,ridder2019traveling}.
These are {nonlocal} traffic flow models.
While {local} models are governed by conservation laws, where the fundamental diagram gives the
relation between flux and density, the flux function of {nonlocal} models depends on an integral evaluation of the density or velocity.
In case of autonomous vehicles, the integration area allows for an interpretation as a connection radius.

Nonlocal traffic flow models have been introduced in~\cite{amorim2015numerical,BlandinGoatin2016}.
Most commonly macroscopic first order models are studied regarding, e.g. existence and well-posedness~\cite{chiarello2018global,friedrich2018godunov,KeimerPflug2017,keimer2018multi,keimer2018bounded}, numerical schemes~\cite{BlandinGoatin2016, ChalonsGoatinVillada2018,friedrich2019maximum,friedrich2018godunov, GoatinScialanga2016}, its singular limit behavior \cite{coclite2020general,colombo2019singular,keimer2019nonlocal} or modeling extensions such as multi-class models~\cite{chiarello2019multiclass}, time delay models~\cite{keimer2019nonlocal}, multilane models \cite{bayen2022multilane,friedrich2020nonlocal} or network formulations~\cite{friedrich2020onetoone,shen2019stationary}.
Recently also microscopic modeling approaches~\cite{friedrich2020micromacro,shen2019stationary,GoatinRossi2017, ridder2019traveling} have been investigated.
In \cite{friedrich2020micromacro}, starting from a microscopic approach a class of {nonlocal} second-order models similar to the {local} generalized Aw-Rascle-Zhang \cite{aw2000resurrection,herty2014datagarz,ZHANG2002traffic} has been derived.
In this work, we will focus on this {nonlocal} second order model.

Even though nonlocal traffic flow models have been studied in various research directions over the years, there are only a few works concerning control problems \cite{ bayen2021boundary, huang2020stability, karafyllis2022control}.
In \cite{bayen2021boundary}, a very general result concerning the exact boundary controllability is obtained.
While in \cite{huang2020stability} a Lyapunov function on a ring road with a linear velocity function is studied.
Also, \cite{karafyllis2022control} considers a Lyapunov function on a ring road, but introduces a different model.
Nevertheless, all these references only cover macroscopic models.

In contrast to that, various {local} traffic flow models have been studied concerning stability and the asymptotic behavior in time: macroscopic models are studied in e.g. \cite{bastin2007lyapunov,blandin2016regularity}, microscopic models in e.g. \cite{cui2017stabilizing,karafyllis2022lyapunov}, while in \cite{karafyllis2022stability} both a microscopic and a corresponding macroscopic model are considered.
For all models stationary solutions are investigated, e.g., all cars move with a constant velocity or the traffic density stays constant.
Therefore, the system is in an equilibrium state.
Analytical results on whether the solution reaches its stationary solution over time are shown.
%A tool to prove this is a Lyapunov function and often an exponential decay in time can be proven.

The aim of this work is to obtain Lyapunov functions and an exponential decay rate for nonlocal second-order models \cite{friedrich2020micromacro}.
As a direct consequence we will also obtain the results for the nonlocal first-order model of \cite{friedrich2018godunov}. 
Since microscopic and macroscopic models are directly related by the number of cars, meaning that for an infinite number of cars, a microscopic model should approximate a macroscopic solution, we study both scales. 
Thereby, we will consider rather general velocity functions.

Note that we only consider microscopic and macroscopic scales.
Future work could also consider the mesoscopic description \cite{visconti2020BGK}.
For nonlocal traffic flow models the mesoscopic description was just recently derived in \cite{chiarello2022macroscopic}.

The upcoming work is structured as follows: Section \ref{sec:models} introduces the considered traffic flow model on the microscopic and  macroscopic scale. Section \ref{sec:stable} contains our main results on stationary solutions, their stability, and asymptotic behavior. Further, we compute explicit rates on the decay rate. In the last section, we present numerical examples which demonstrate the theoretical results.
Further, the numerical results probably hold under less restrictive conditions than we assume in the theory.

\section{Nonlocal second order traffic flow models}\label{sec:models}

We recall the nonlocal second order traffic flow model considered in \cite{friedrich2020micromacro}. 

\subsection{The microscopic model}
On the microscopic scale, we consider the description of $N+1$ individual drivers.
First-order nonlocal microscopic traffic flow models were introduced in \cite{GoatinRossi2017, ridder2019traveling}.
Here, the speed of the $i-$th driver depends on a weighted mean downstream velocity (or downstream density).
On  average only cars are considered which are at most $\ndt>0$ far away from the driver $i$.
If $x_i(t)$ is the position of driver $i$ at time $t$, all cars in the interval $[x_i(t),x_i(t)+\ndt]$ are taken into account for the driver $i$ to decide about the speed.
Furthermore, we assume that each driver has an individual empty road velocity that does not change over time and is expressed by a Lagrangian marker $\lm_i$.
This results in the following microscopic equations for $N+1$ vehicles considered in \cite{friedrich2020micromacro}:
\begin{align}
\label{eq:micronewnotation}
\begin{cases}
    \frac{d}{dt} x_i(t) =\sum_{j=0}^{\Ne-1} \gamma_{i,j}(t) V_{i,j}(t),&\qquad i=0,\dots,N-1,\\[5pt]
    \frac{d}{dt} x_N(t)=\bar v,\\[5pt]
    \frac{d}{dt} \lm_i(t)=0,&\qquad i=0,\dots,N-1,
    \end{cases}
\end{align}
where the velocity $V_{i,j}$ is defined as follows
\begin{equation*}
    V_{i,j}(t)=\begin{cases}
        v\left(\frac{1}{N(x_{i+1+j}(t)-x_{i+j}(t))},\lm_{i+j}(t)\right), &\text{ for }i+j\leq N-1\\[5pt]
        v\left(0,\bar v\right)=\bar v, &\text{ for }i+j> N-1,
    \end{cases}
\end{equation*}
with positive weights satisfying $\sum_{j=0}^{\Ne-1} \gamma_{i,j}(t)=1$ for all $i=0,\dots,N-1$ and $t>0$, and
where $x_i(t)$ represents the position of the $i$-th car at time $t$, $N+1$ is the number  of cars, $v$ is a suitable velocity function and $\bar v>0$ is the velocity of the leader.
Here, $\Ne$ is chosen in such a way that $x_{i+\Ne}(t)-x_i(t)\geq \ndt$ with $\ndt>0$, e.g. $\Ne$ can be the smallest integer such that $\Ne \geq (N+1)\ndt$.
We note that there is a degree of freedom in modeling the leading vehicle.
We will later comment on the choice of $\bar v$, as the velocity of the leading vehicle can be used as a control input to the system.

For strictly positive constants $\lm_0,\dots,\lm_{N-1}$, we write the model as follows
\begin{align}
\label{eq:microGARZ}
    \frac{d}{dt} x_i(t) =\sum_{j=0}^{\Ne-1} \gamma_{i,j}(t) V_{i,j}(t),&\qquad i=0,\dots,N
\end{align}
and the velocity $V_{i,j}$ simplifies to
\begin{equation*}
    V_{i,j}(t)=\begin{cases}
        v\left(\frac{1}{N(x_{i+1+j}(t)-x_{i+j}(t))},\lm_{i+j}\right), &\text{ for }i+j\leq N-1\\
        \bar v, &\text{ for }i+j> N-1.
    \end{cases}
\end{equation*}

\begin{remark}
If the Lagrangian markers coincide $\lm_i=\bar v$ for $i=0,\dots,N-1$, the microscopic model reduces to the first order model considered in \cite{ridder2019traveling}.
\end{remark}

Let us now specify the precise formulas for the weights $\gamma_{i,j}$ and the assumptions on $v$.
To define the weights we introduce the kernel function $\twt$ which is defined as
\begin{equation}
\label{eq:kernel}
\twt(x)=\begin{cases}
\wt(x),&\text{if } x\in[0,\ndt],\\
0,&\text{else}.
\end{cases}
\end{equation}
Here, the assumptions on $\wt$ are
\begin{align} \label{eq:ass:kernel}
\wt \in C^1([0,\ndt];\R^+)\  & \text{with} \ \wt'\leq 0 \text{,}\ \int_0^\ndt \wt(x) \dx=1 \ \forall\ \ndt>0.
\end{align}
Now we define the relation between the kernel function $\twt$ and the weights $\gamma$.
It is expressed by the following equation:
\begin{align}\label{eq:microkernel}
    \gamma_{i,j}(t):=\begin{cases}
     \int_{x_{i+j}(t)}^{x_{i+j+1}(t)}\twt(y-x_i(t))\dy,\quad &\text{if } i+j\leq N-1,\\[5pt]
    \int_{\min\{x_{N}(t),x_i(t)+\ndt\}}^{x_{i}(t)+\ndt}\twt(y-x_i(t))\dy,\quad&\text{if }i+j= N,\\[5pt]
    0,&\text{if }i+j>N.
    \end{cases}
\end{align}
Due to the definition of $\twt$ in \eqref{eq:kernel} the weights can be zero for $j$ large, even if $i+j\leq N$.
In particular, the weights are the evaluation of the function $\twt$ over a grid given by the position of the individual cars.
By definition of $\twt$, it is obvious that $\sum_{j=0}^{\Ne-1} \gamma_{i,j}(t)=1$ for all $i=0,\dots,N-1$ and $t>0$ holds.

Next, we state the assumptions on the velocity function.
\rv{We assume $(\rho,\omega) \to v(\rho,\omega)$ to be a  twice continuously differentiable function in $\rho$ and a continuously differentiable function in $\omega$.}
As in \cite{friedrich2020micromacro,herty2014datagarz}, the following \rv{additional} assumptions on the velocity function $v(\rho,\lm)$ are considered:
\begin{subequations}\label{eq:assumptionsvelocity}
    \begin{align} \label{hyp:velocityinzero}
    &v (\rho,\lm)\geq 0, \quad v (0,\lm) = \lm, \quad v (\rho, 0) = 0, \\ \label{eq:velocitydecreasing}
    &\text{for }f(\rho,\lm)=\rho v(\rho,\lm)\text{ we have }\frac{\partial^2 f}{\partial \rho^2}(\rho,\lm) < 0\text{ for }\lm > 0 ,\\ \label{eq:lagrangiandecreasing}
    &\frac{\partial v}{\partial \lm}(\rho,\lm)>0.
    \end{align}
\end{subequations}
The first assumption in \eqref{hyp:velocityinzero} ensures that vehicles never travel backward, while the second one shows why $\lm$ can be interpreted as the empty road velocity. 
Condition \eqref{eq:velocitydecreasing}  implies $\frac{\partial v}{\partial \rho}(\rho,\lm)<0$ for $\lm>0$, \rv{since} $v$ is a $C^2$ function in $\rho$, see also \cite[Lemma 1]{herty2014datagarz}. 
The assumption \eqref{eq:lagrangiandecreasing} implies that a faster empty road velocity results in a faster velocity for all possible densities.
Additionally to \eqref{eq:assumptionsvelocity} we assume
 \begin{equation}\label{eq:assumptionmaximaldensity}
   \forall\ \lm>0 \quad \exists\ \rho^{\max}_\lm>0:\quad  v(\rho^{\max}_\lm,\lm)=0. 
 \end{equation}

The latter assumption is needed to ensure a maximum principle.
\begin{remark}
Note that the maximum principle in \eqref{eq:microGARZ} is kept and cars are not overtaking each other for every $\bar v>0$.
In \cite[Proposition 1]{friedrich2020micromacro} the maximum principle is only proven for the case $\bar v\geq \lm_{N-1}$.
Nevertheless, the result can be extended.
It is obvious that in the case of $\bar v< \lm_{N-1}$ the proof of \cite[Proposition 1]{friedrich2020micromacro} remains valid for $i=1,\ldots,N-2$.
Assume that as in the proof of \cite[Proposition 1]{friedrich2020micromacro} we are at a time $t$ at which the distance of car $N-1$ to the $N$-th car is minimal.
Then, we have
\begin{align*}
\frac{d}{dt}\left(x_N(t)-x_{N-1}(t)\right)=&\bar v-\gamma_{N-1,0}\rv{(t)}v\left(\frac{1}{N\left(x_N(t)-x_{N-1}(t)\right)},\lm_{N-1}\right)-(1-\gamma_{N-1,0}\rv{(t)})\bar v\\
=& \gamma_{N-1,0}\rv{(t)}\left(\bar v-v\left(\frac{1}{N\left(x_N(t)-x_{N-1}(t)\right)},\lm_{N-1}\right)\right)= \gamma_{N-1,0}\rv{(t)}\bar v\geq 0.
\end{align*}
Due to the minimum distance between the $N$-th and $N-1$-th car and Assumption \eqref{eq:assumptionmaximaldensity} on $v$ the last equality holds.
Therefore, the distance is increasing and the cars are not overtaking each other.
\end{remark}

\subsection{The macroscopic model}
In \cite{friedrich2020micromacro} it is proven that the solution of the microscopic model \eqref{eq:microGARZ} converges for $N\to\infty$ towards the weak solution of the following nonlocal system \rv{for $(t,x)\in \R^+\times \R$}:
\begin{equation}\label{eq:macrosystem}
    \begin{cases}
\partial_t    \rho+\partial_x\left(\rho\, V(t,x) \right) =0,\\[3pt]
   \partial_t \lm+ V(t,x)\, \partial_x\lm=0,
    \end{cases}
\end{equation}
where 
\begin{equation}\label{eq:NV:conv}
V(t,x):=\left(\wt \ast v(\rho,\lm)\right) (t,x) = \int_x^{x+\ndt} \wt(y-x)v(\rho(t,y),\lm(t,y))\dy \quad \ndt>0. 
\end{equation}
Note that the flux depends explicitly on $x$ through the convolution product.
The notation $V(t,x)$ is just used as an abbreviation.
Here, $\wt$ fulfills the assumptions \eqref{eq:ass:kernel} and $v$ the assumptions \eqref{eq:assumptionsvelocity} and \eqref{eq:assumptionmaximaldensity}.
The model is strongly inspired by the {local} GARZ model \cite{herty2014datagarz}.
This model generalizes the relations between density, Lagrangian marker, and velocity, to ensure a unique maximum density, which is not usual for most of the commonly known second-order traffic flow models such as the ARZ model.
Hence, we will refer to \eqref{eq:macrosystem} as the nonlocal GARZ model.
The equation needs to be accompanied by the initial conditions: 
\begin{equation}\label{eq:initcond}
\begin{aligned}
 &\rho(0,x)=\rho_0(x)\in \BV(\R),\ \rho_0(x)\leq \rho_{\max}:=\max_{\lm} \rho^\lm_{\max},\ \spt(\rho_0) \text{ compact and connected},\\
 &\lm(0,x)=\lm_0(x)\in \BV(\R).
 \end{aligned}
\end{equation}
The last assumption on $\rho_0$ ensures that the initial density for $x\in \spt(\rho_0)$ is strictly greater than zero.
This assumption is needed to prove the limit from the microscopic to the macroscopic model to ensure that cars are (in the limit) always present on the road.

\begin{remark}
Similar to the microscopic case, the model \eqref{eq:macrosystem} reduces to the first order model considered in \cite{friedrich2018godunov} by choosing a constant Lagrangian marker $\lm(t,x)=\lm\ \forall (t,x)\in \R^+\times \R$.
\end{remark}

Setting $q=\rho \lm$, the model can be written in conservative form for $\rho \neq 0$ as
\begin{equation}\label{eq:macroconservative}
    \begin{cases}
    \partial_t\rho+\partial_x\left(\rho\, V(t,x) \right) =0\\[3pt]
   \partial_t q+ \partial_x\left( q\, V(t,x) \right) =0,
    \end{cases}
\end{equation}
where the convolution product is defined as above and $\lm=q/\rho$.
\section{Lyapunov stabilization}\label{sec:stable}
\subsection{Stability of the microscopic models} \label{sec:firstorderstability}
For deducing the stability of the nonlinear ordinary differential equations \eqref{eq:microGARZ} we will consider the corresponding linearized ODE first.

To determine a linearized ODE we define the distance between two cars
\begin{equation}
y_i(t):=x_{i+1}(t)-x_i(t),\quad \text{for } i=0,\ldots,N-1.
\end{equation}
As the Lagrangian marker is directly given by the initial conditions we consider the system \eqref{eq:micronewnotation}.
Then, we obtain an equilibrium point, if the distance between two cars is constant over the time, i.e. $y_i(t)=\bar L_i$ for some $\bar L_i>0$ and $i=0,\dots,N-1$.
A constant distance between all pairs of cars means that all cars move at the same speed, which will be denoted as the equilibrium speed $\bar v$.
There is a direct relation to compute the distance $\bar L_i$ from a given equilibrium velocity $\bar v$, i.e.
\begin{align*}
v\left(\frac{1}{N\bar L_i},\lm_i\right)=\bar v\ \forall\ i=0,\dots,N-1.
\end{align*}
Hence, for a suitable and given equilibrium velocity we can compute the equilibrium distance $\bar L_i$.
As the Lagrangian marker is the maximum velocity of each individual driver, we have a restriction on $\bar v$, i.e. $\bar v \leq \min_{i=1,\ldots,N-1}\lm_i$ needs to hold.
Further, we need to impose the following assumption:
\begin{assumption}\label{ass:vprime}
We assume that  an upper bound on the derivative of $v$ with respect to $\rho$ exists, i.e.
        $$ \max_{\lm_i=\{\rv{0},\dots,N-1\}}\partial_\rho v(\rho,\lm_i)\leq v'_{\max}<0.$$
\end{assumption}
Next, we state the following result for the stability of the equilibrium states.
\begin{theorem}\label{prop:microstable}
Under Assumption \ref{ass:vprime} the stationary solution $y_i(t)=\bar L_i$, $i=0,\dots,N-1$ for the dynamics given by \eqref{eq:micronewnotation} is (locally) asymptotically stable.
\end{theorem}
\begin{proof}
Let us at first consider the cars $i$ for which $x_i(t)<x_N(t)-\ndt$ holds, i.e. the cars which do not see the leading vehicle.
We denote those cars by $i=0,\dots,\mathcal{J}-1$\footnote{Note that $\mathcal{J}=N$ is possible here, then we only have to consider the cars $i=0,\dots,\mathcal{J}-1$.}.
Since $\frac{d}{dt}y_i(t)=\frac{d}{dt}(x_{i+1}(t)-x_i(t))=0$, we obtain
\begin{align*}
\frac{d}{dt}y_i(t)=&\sum_{j=0}^{\Ne-1} \gamma_{i+1,j}(t) v\left(\frac{1}{Ny_{i+j+1}(t)},\lm_{i+j+1}\right)- \gamma_{i,j}(t) v\left(\frac{1}{Ny_{i+j}(t)},\lm_{i+j}\right)\\
=& \sum_{j=1}^{\Ne-1} \left(\gamma_{i+1,j-1}(t)-\gamma_{i,j}(t)\right) v\left(\frac{1}{Ny_{i+j}(t)},\lm_{i+j}\right)- \gamma_{i,0}(t) v\left(\frac{1}{Ny_{i}(t)},\lm_{i}\right)\\
&+\gamma_{i+1,\Ne-1}(t) v\left(\frac{1}{Ny_{i+\Ne}(t)},\lm_{i+\Ne}\right).
\end{align*}
Now, we need to determine the weights \eqref{eq:microkernel} depending on $y_i(t)$:
\begin{align*}
\gamma_{i,j}(t)=& \int_{x_{i+j}(t)}^{x_{i+j+1}(t)}\twt(y-x_i(t))\dy
= \int_{x_{i+j}(t)-x_i(t)}^{x_{i+j+1}(t)-x_i(t)}\twt(y)\dy
= \int_{\sum_{k=0}^{j-1}y_{i+k}(t)}^{\sum_{k=0}^{j}y_{i+k}(t)}\twt(y)\dy.
\end{align*}
This gives us
\begin{align*}
F_i(y):=&\frac{d}{dt}y_i(t)\\
=& \sum_{j=1}^{\Ne-1} \left(\int^{\sum_{k=1}^{j}y_{i+k}(t)}_{\sum_{k=1}^{j-1}y_{i+k}(t)}\twt(y)\dy-\int^{\sum_{k=0}^{j}y_{i+k}(t)}_{\sum_{k=0}^{j-1}y_{i+k}(t)}\twt(y)\dy\right) v\left(\frac{1}{Ny_{i+j}(t)},\lm_{i+j}\right)\\
&- \int_{0}^{y_{i}(t)}\twt(y)\dy v\left(\frac{1}{Ny_{i}(t)},\lm_{i}\right)+\int^{\sum_{k=1}^{\Ne-1}y_{i+k}(t)}_{\sum_{k=1}^{\Ne-2}y_{i+k}(t)}\twt(y)\dy v\left(\frac{1}{Ny_{i+\Ne}(t)},\lm_{i+\Ne}\right).
\end{align*}
As we want to consider a linear ODE by linearizing $F$ around the equilibrium $y_i=\bar L_i$ for all $i=0,\dots,N-1$, we first compute for $i=1,\dots,\mathcal{J}-1$
\begin{align}\notag
\rv{\partial_{y_i}F_i(y)}
=& -\sum_{j=1}^{\Ne-1} \left(\twt\left({\sum_{k=0}^{j}y_{i+k}(t)}\right)-\twt\left({\sum_{k=0}^{j-1}y_{i+k}(t)}\right)\right) v\left(\frac{1}{Ny_{i+j}(t)},\lm_{i+j}\right)\\
&- \twt(y_{i}(t)) v\left(\frac{1}{Ny_{i}(t)},\lm_{i}\right)+\int_{0}^{y_{i}(t)}\twt(y)\dy\, \partial_{\rho}v\left(\frac{1}{Ny_{i}(t)},\lm_{i}\right)\frac{1}{Ny_{i}(t)^2}\notag
\end{align}
Now we repeat the calculations above for $i=\mathcal{J},\dots,N-1$ and obtain:
\begin{align}
\notag  F_i(y)=&\sum_{j=0}^{N-2-i} \gamma_{i+1,j}(t)V_{i+1,j}(t)+ \left(1-\sum_{j=0}^{N-2-i} \gamma_{i+1,j}(t)\right) \bar v\\
\notag  &-\left(\sum_{j=0}^{N-1-i} \gamma_{i,j}(t)V_{i,j}(t)+ \left(1-\sum_{j=0}^{N-1-i} \gamma_{i,j}(t)\right) \bar v\right)\\
\label{eq:detailsidentity}    =&\sum_{j=1}^{N-1-i} (\gamma_{i+1,j-1}(t)-\gamma_{i,j}(t))\left(v\left(\frac{1}{Ny_{i+j}},\lm_{i+j}\right)-\bar v\right)-\gamma_{i,0}(t)\left(v\left(\frac{1}{Ny_{i}},\lm_{i}\right)-\bar v\right)
\end{align}
and
\begin{align}\notag
\partial_{y_i}F_i(y)=& -\sum_{j=1}^{\Ne-1} \left(\wt\left({\sum_{k=0}^{j}y_{i+k}(t)}\right)-\wt\left({\sum_{k=0}^{j-1}y_{i+k}(t)}\right)\right) \left(v\left(\frac{1}{Ny_{i+j}(t)},\lm_{i+j}\right)-\bar v\right)\\
&- \wt(y_{i}(t)) \left(v\left(\frac{1}{Ny_{i}(t)},\lm_{i}\right)-\bar v\right)+\int_{0}^{y_{i}(t)}\wt(y)\dy\, \partial_{\rho}v\left(\frac{1}{Ny_{i}(t)},\lm_{i}\right)\frac{1}{Ny_{i}(t)^2}\notag.
\end{align}
Note that we are not dealing with the case $\gamma_{i,j}(t)=0$ for some $i,j$ and hence we have $\wt$ instead of $\twt$.

Since the cars are only forward looking the Jacobi matrix of $F$ is a upper triangular matrix with eigenvalues given by the entries of its diagonal.
Hence, we set $y_i=\bar L_i$ to obtain the eigenvalues of the linearized ODE, i.e. 
for $i=0,\dots,\mathcal{J}-1$
\begin{align*}
\partial_{y_i}F_i(\bar L)=& -\sum_{j=1}^{\Ne-1} \left(\twt\left({\sum_{k=0}^{j}\bar L_{i+k}}\right)-\twt\left({\sum_{k=0}^{j-1}\bar L_{i+k}}\right)\right) \bar v\\
&- \twt(\bar L_{i}) \bar v+\int_{0}^{y_{i}(t)}\twt(y)\dy\, \partial_{\rho} v\left(\frac{1}{N \bar L_{i}},\bar \lm_{i}\right)\frac{1}{N\bar L_{i}^2}\\
=& \left(\twt\left({\bar L_{i}}\right)-\twt\left({\sum_{k=0}^{\Ne-1}\bar L_{i+k}}\right)\right) \bar v- \twt(\bar L_{i}) \bar v\\
&+\int_{0}^{\bar L_{i}}\twt(y)\dy\, \partial_{\rho} v\left(\frac{1}{N \bar L_{i}},\bar \lm_{i}\right)\frac{1}{N\bar L_{i}^2}\\
=&-\twt\left({\sum_{k=0}^{\Ne-1}\bar L_{i+k}}\right) \bar v+\int_{0}^{L_{i}}\twt(y)\dy\, \partial_{\rho} v\left(\frac{1}{N \bar L_{i}},\bar \lm_{i}\right)\frac{1}{N\bar L_{i}^2}<0
\end{align*}
and for $i=\mathcal{J},\ldots,N-1$
\begin{align*}
\partial_{y_i}F_i(\bar L)=& \int_{0}^{L_{i}}\wt(y)\dy\, \partial_{\rho} v\left(\frac{1}{N \bar L_{i}},\bar \lm_{i}\right)\frac{1}{N\bar L_{i}^2}<0
\end{align*}
due to the assumptions on $\partial_\rho v$.
As the fixed point of the linear ODE is asymptotically stable, so is the fixed point of the nonlinear ODE.
\end{proof}

After having seen that the microscopic traffic model can be stabilized, we define an appropriate Lyapunov function and determine the rate of convergence towards the steady state.

\subsection{Lyapunov function for the density}
We will consider a Lyapunov function measuring the $L^2$ distance of the density.
%For the theoretical result we need to restrict ourselves to constant kernel functions.
%Hence, we impose the following assumption:
%\begin{assumption}\label{ass:kernelcons}
%We set the kernel function to

\subsubsection{The microscopic level}
We consider the microscopic system \eqref{eq:microGARZ} and a given equilibrium velocity $\bar v\leq \min_{i=\rv{0},\ldots, N-1} \lm_i$.
In particular, this velocity of the leading vehicle should control the system towards the corresponding equilibrium.
As the leading vehicle has only an  influence on the cars  in the distance $\ndt$ behind it, we cannot prove a stabilization result for all cars.
\rv{
For a set of specific cars we are able to obtain the following stronger maximum principle:
\begin{lemma}[Maximum principle]\label{lem:max}
Let the initial placement of cars and the equilibrium velocity $\bar v$ be chosen such that $J\in\{0,\dots,N-1\}$ being the smallest integer satisfying
\begin{align}\label{eq:sufficient}
\sum_{i=J}^{N-1} \max\lbrace y_i(0),\bar L_i\rbrace \leq \ndt
\end{align}
exists.
Further, let Assumption \ref{ass:vprime} hold and 
the dynamics given by equation \eqref{eq:microGARZ}.
In addition, we assume either
\begin{enumerate}[label=\alph*)]
    \item the kernel function is constant, i.e. $\wt(x)=\frac{1}{\ndt},$
    or 
    \item the kernel function is concave, $\lm_i=\lm$ for $i=J,\ldots,N-1$, $\lm>0$ and the initial data satisfies either $y_J(0)\geq\dots\geq y_{N-1}(0)\geq \bar L$ or $y_J(0)\leq\dots\leq y_{N-1}(0)\leq \bar L$ with $v(\bar L,\lm)=\bar v$.
\end{enumerate}
Then, a maximum principle holds, i.e.,
\begin{equation}\label{eq:maxmicro}
y_i(t)\in [\min\{y_i(0),\bar L_i\},\max\{y_i(0),\bar L_i\}]\ \forall t>0,\ i=\mathcal{J} ,\dots,N-1.
\end{equation}
\end{lemma}
\begin{proof}
We define 
\begin{equation}\label{eq:bulk}
\BB:=\{i\in \{1,\ldots,N-1\} |x_i(t)\in [x_N(t)-\ndt,x_N(t)]\ \forall t\geq 0\}.
\end{equation}
Here, $\BB$ is the set of cars that are at maximum $\ndt$ behind the leading vehicle for all times $t\geq 0$.
Note that the set $\BB$ may  be empty.
If it is not empty, we set $\mathcal{J}:=\min \BB$.
%%%{ \textcolor{red}\Large What happens when this set is empty? }
This gives the integer of the car being the ''last'' car which is always in the interval $[x_N(t)-\ndt,x_N(t)]$.
Now, we are able to provide a stricter maximum principle with respect to the distance and the microscopic density, respectively, for the set $\BB$.\\
Let us assume $\BB\neq \emptyset$.
We prove the claim by induction and starting at car $N-1$:
\begin{align*}
    \frac{d}{dt}\left(y_{N-1}(t)-\bar L_{N-1}\right)=&\bar v- \gamma_{N-1,0}(t) v\left(\frac{1}{Ny_{N-1}(t)},\lm_{N-1}\right)-(1-\gamma_{N-1,0}(t))\bar v\\
    =& \gamma_{N-1,0}(t)\left(v\left(\frac{1}{N\bar L_{N-1}},\lm_{N-1}\right)-  v\left(\frac{1}{Ny_{N-1}(t)},\lm_{N-1}\right)\right),\\
    \intertext{where we use the definition of the equilibrium velocity. Now, we apply the mean value theorem: once for $v(\cdot,\lm_{N-1})$ and once for $\gamma_{N-1,0}(t)=\int_0^{y_{N-1}(t)} \wt(y)dy=\wt(\zeta_{N-1}(t))\,y_{N-1}(t)$. This leaves us with}
    =&\wt(\zeta_{N-1}(t))\,y_{N-1}(t)\,\partial_\rho v(\xi_{N-1}(t),\lm_{N-1}) \left(\frac{1}{N\bar L_{N-1}}-\frac{1}{Ny_{N-1}(t)}\right)\\
    =&\wt(\zeta_{N-1}(t))\,\partial_\rho v(\xi_{N-1}(t),\lm_{N-1}) \frac{1}{N\bar L_{N-1}}\left(y_{N-1}(t)-\bar L_{N-1}\right).
\end{align*}
A solution to the ordinary differential equation is given by
\begin{align}\label{eq:dNm1}
    y_{N-1}(t)-\bar L_{N-1}=\left(y_{N-1}(0)-\bar L_{N-1}\right)\exp\left(\frac{\int_0^t \wt(\zeta_{N-1}(s))\, \partial_\rho v(\xi_{N-1}(s),\lm_{N-1}) ds}{N\bar L_{N-1}}\right).
\end{align}
Due to the negative sign of the derivative of $v$ we bound  the exponential term by one. 
Let us consider $y_{N-1}(0)\leq \bar L_{N-1}$, then \eqref{eq:dNm1} is estimated from above by zero, such that $y_{N-1}(t)\leq \bar L_{N-1}$ holds.
Further, we estimate the term from below by $y_{N-1}(0)-\bar L_{N-1}$.
Hence, $y_{N-1}(t)\in[y_{N-1}(0),\bar L_{N-1}]$ holds.
The case $y_{N-1}(0)>\bar L_{N-1}$ works analogously.
This proves the claim for the car $N-1$:
\[y_{N-1}(t)\in [\min\{y_{N-1}(0),\bar L_{N-1}\},\max\{y_{N-1}(0),\bar L_{N-1}\}].\]
Now, we consider a car $i\in\{\mathcal{J},\ldots,N-2\}$ and suppose the claim holds for all cars $j=i+1,\ldots,N-1$.
Following the calculations to obtain \eqref{eq:detailsidentity} in Theorem \ref{prop:microstable}, we obtain
\begin{align}\label{eq:dygeneral}
    \frac{d}{dt}\left(y_i(t)-\bar L_i\right)=&\sum_{j=1}^{N-1-i} (\gamma_{i+1,j-1}(t)-\gamma_{i,j}(t))\left(v\left(\frac{1}{Ny_{i+j}\rv{(t)}},\lm_{i+j}\right)-\bar v\right)\\
    \nonumber &-\gamma_{i,0}(t)\left(v\left(\frac{1}{Ny_{i}\rv{(t)}},\lm_{i}\right)-\bar v\right).
\end{align}
We need to distinguish the two cases a) and b) and start with the consideration of the  constant kernel. In this case,
\begin{align}\label{eq:dyconstant}
    \frac{d}{dt}(y_i(t)-\bar L_i)=-\frac{y_i(t)}{\ndt}\left(v\left(\frac{1}{Ny_{i}(t)},\lm_{i}\right)-\bar v\right).
\end{align}
Again, using the mean value theorem, where $\xi_i(t)$ denotes the corresponding value between $1/(Ny_i(t))$ and $1/(N\bar L_i)$, we see that
\begin{align*}
    \frac{d}{dt}(y_i(t)-\bar L_i)=\frac{1}{N\bar L_i\, \ndt}\partial_\rho v\left(\xi_i(t),\lm_{i}\right)(y_i(t)-\bar L_i),
\end{align*}
such that we obtain
\begin{align*}
    y_i(t)-\bar L_i=(y_i(0)-\bar L_i)\exp\left(\frac{1}{N\bar L_i\, \ndt}\int_0^t \partial_\rho v\left(\xi_i(s),\lm_{i}\right)ds\right).
\end{align*} 
Following the same arguments as for the car $N-1$, we  deduce the maximum principle \eqref{eq:maxmicro}.
This finishes the proof for the case a) and all cars which are contained in the set $\mathcal{B}.$\\
The proof of the maximum principle in the case b) is more involved.
For simplicity, we only outline the main steps for the case $y_J(0)\leq\dots\leq y_{N-1}(0)\leq \bar L$.
The details can be found in Lemma \ref{lem:maxappendix} in the Appendix.
\begin{enumerate}
    \item We start by proving the upper bound $\bar L$ on $y_i(t)$.
    \item Using the upper bound, we  show that the monotonicity of the initial data is preserved, if the kernel function is concave.
    \item Finally, this allows to prove the lower bound $y_i(0)\leq y_i(t)$.
\end{enumerate}
In case of monotone decreasing initial data,  a similar proof applies.  Hence, we have established the maximum principle under the assumption that $\BB\neq\emptyset$ and $J\geq \mathcal{J}$. 
It remains to prove that those assumptions are valid.
The upper bounds on the distances $y_i(t)$ for $i=\mathcal{J},\dots,N-1$ yield
\begin{align*}
\sum_{i=\mathcal{J}}^{N-1} y_i(t)\leq \sum_{i=\mathcal{J}}^{N-1} \max\{y_i(0),\bar L_i\}.
\end{align*}
This motivates the choice of the initial conditions in \eqref{eq:sufficient}.
In fact, $J$ is an a priori upper bound on $\mathcal{J}$.
Hence, under condition \eqref{eq:sufficient} also $\mathcal{J}$ exists and $\BB\neq \emptyset$ is guaranteed.
\end{proof}
\begin{remark}
Under the assumptions of Lemma \ref{lem:max} we obtain a lower bound on the microscopic density for the cars behind the leading vehicle, i.e.,
\begin{align}\label{eq:rhominmicro}
   \rho_{\min}:=\min_{i=J,\ldots,N} \min\left\{\frac{1}{Ny_i(0)},\frac{1}{N\bar L_i}\right\}.  
\end{align}
\end{remark}
} 
%%%%%%%%%%%%
\rv{
This maximum principle is one of the keys to prove the following stabilization results:
\begin{theorem}[Lyapunov stabilization: constant kernel]\label{thm:microconstant}
Let the initial placement of cars be as in Lemma \ref{lem:max}.
Further, let Assumption \ref{ass:vprime} hold and
the dynamics given by equation \eqref{eq:microGARZ}.
The kernel function shall be constant, i.e. $\wt(x)=\frac{1}{\ndt}$.
We define the Lyapunov function
\begin{align}\label{eq:microLyapunov}
    L(t):=\sum_{i=J}^{N-1}y_i(t)\left( \frac{1}{N y_i(t)}-\frac{1}{N \bar L_i}\right)^2.
\end{align}
Then, we have
\begin{align*}
    L(t)\leq L(0) \exp\left(\frac{2}{\ndt} v'_{\max}\, \rho_{\min}\,t\right) \ \forall t\geq 0,
\end{align*}
where $\rho_{\min}$ is given by \eqref{eq:rhominmicro} and $\bar L_i$ is the equilibrium distance $v\left(\frac{1}{N \bar L_i},\lm_i\right)=\bar v$.
\end{theorem}
\begin{proof}
\rv{The time derivative of the Lyapunov function $L(t)$ is given by}
\begin{align*}
    \frac{d}{dt}L(t)=&\sum_{i=J}^{N-1}2y_i(t)\left( \frac{1}{N y_i(t)}-\frac{1}{N \bar L_i}\right)\left(-\frac{1}{Ny_i(t)^2}\right)\frac{d}{dt}y_i(t)+\left( \frac{1}{N y_i(t)}-\frac{1}{N \bar L_i}\right)^2\frac{d}{dt}y_i(t)\\
    =&\sum_{i=J}^{N-1}\left( \frac{1}{N y_i(t)}+\frac{1}{N \bar L_i}\right)\left( \frac{1}{N \bar L_i}-\frac{1}{N y_i(t)}\right)\frac{d}{dt}y_i(t).
    \intertext{\rv{By plugging in the derivative \eqref{eq:dyconstant} of $y_i(t)$ as well as using the mean value theorem with the corresponding value $\xi_i(t)$ we obtain}}
    =& \frac{1}{\ndt}\sum_{i=J}^{N-1} \partial_\rho v(\xi_i\rv{(t)},\lm_i)\left( \frac{1}{N y_i(t)}+\frac{1}{N \bar L_i}\right)y_i(t)\left( \frac{1}{N y_i(t)}-\frac{1}{N \bar L_i}\right)^2.\\
    \intertext{\rv{Finally, we estimate this using the maximum principle}}
    \leq &\frac{2}{\ndt} v'_{\max}\, \rho_{\min} L(t).
\end{align*}
Applying Gr\"onwall's inequality yields
\begin{align*}
    L(t)\leq L(0) \exp\left(\frac{2}{\ndt} v'_{\max}\, \rho_{\min}t\right).
\end{align*}
Note that the rate is negative due to the sign of $v'_{\max}$.
\end{proof}
We note that a constant convolution kernel can, e.g. model connected autonomous vehicles, which have the same degree of accuracy on information about the downstream traffic, independent of the distance.
In the case of non-constant kernels the accuracy of information decreases with the distance.
For concave kernels we obtain a similar stabilization result:
\begin{theorem}[Lyapunov stabilization: concave kernel]\label{thm:microconcave}
Let the initial placement of cars be as in Lemma \ref{lem:max} and additionally we assume $\lm_i=\lm$ for $i=J,\ldots,N-1$ and the initial datum satisfies either $y_J(0)\geq\dots\geq y_{N-1}(0)\geq \bar L$ or $y_J(0)\leq\dots\leq y_{N-1}(0)\leq \bar L$ with $v(\bar L, \lm)=\bar v$.
Further, let Assumption \ref{ass:vprime} hold and
the dynamics given by equation \eqref{eq:microGARZ}.
The kernel function shall be concave.
Then, we obtain for the Lyapunov function \eqref{eq:microLyapunov}
\begin{align*}
    L(t)\leq L(0) \exp\left(2\, v'_{\max}\,  \rho_{\min}\int_0^t \wt(x_N(s)-x_J(s))ds\right) \ \forall\, t\geq 0.
\end{align*}
\end{theorem}
\begin{proof}
As in the proof of Theorem \ref{thm:microconstant}, we have 
\begin{align*}
    \frac{d}{dt}L(t)=&\sum_{i=J}^{N-1}\left( \frac{1}{N y_i(t)}+\frac{1}{N \bar L}\right)\left( \frac{1}{N \bar L}-\frac{1}{N y_i(t)}\right)\frac{d}{dt}y_i(t).
\end{align*}
We start again with considering the case $y_J(0)\leq\dots\leq y_{N-1}(0)\leq \bar L$.
For simplicity, we drop the dependence of $v$ on $\lm$.
Starting form the derivative \eqref{eq:dygeneral} and adding a zero, we obtain
\begin{align*}
\frac{d}{dt}y_i(t)=&\sum_{j=1}^{N-1-i} (\gamma_{i+1,j-1}(t)-\gamma_{i,j}(t))\left(v\left(\frac{1}{Ny_{i+j}(t)}\right)-v\left(\frac{1}{Ny_{i}(t)}\right)\right)\\
      \nonumber &+\left(\sum_{j=1}^{N-1-i} \gamma_{i+1,j-1}(t)-\sum_{j=0}^{N-1-i}\gamma_{i,j}(t)\right)\left(v\left(\frac{1}{Ny_{i}(t)}\right)-\bar v\right).
      \intertext{We use that the monotonicity is kept (see the proof of Lemma \ref{lem:maxappendix}) and the definition of the weights to get}
    \geq &-\int_{x_i(t)}^{x_{i+1}(t)}\wt(x_N-y)dy\left(v\left(\frac{1}{Ny_{i}(t)}\right)-\bar v\right)= y_i(t)\wt(\zeta^N_i(t)) \left(\bar v-v\left(\frac{1}{Ny_{i}(t)}\right)\right),
    \intertext{where $\zeta^N_i(t)\in[x_N(t)-x_{i+1}(t),x_N(t)-x_i(t)]$. We apply the mean value theorem in $v$ with the corresponding value $\xi_i(t)$ to obtain}
    \geq& -y_i(t)\wt(\zeta^N_i(t))\,\partial_\rho v\left(\xi_i(t)\right)\left(\frac{1}{N y_i(t)}-\frac{1}{N\bar L}\right).
\end{align*}
Due to the maximum principle we can estimate
\begin{align*}
    \frac{d}{dt}L(t)\leq &\sum_{i=J}^{N-1}\left( \frac{1}{N y_i(t)}+\frac{1}{N \bar L}\right)\wt(\zeta^N_i(t))\,\partial_\rho v\left(\xi_i(t)\right)y_i(t)\left(\frac{1}{N y_i(t)}-\frac{1}{N\bar L}\right)^2\\
    \leq & 2 \wt(x_N(t)-x_J(t)) v'_{\max}\, \rho_{\min} L(t).
\end{align*}
In the case that the initial distances are decreasing we can proceed analogously.
Further, applying Grönwall's inequality finishes the proof.
\end{proof}
}
\rv{
We note that the case of increasing distances which are lower than the equilibrium distance is more relevant for applications.
For instance, think of a traffic jam in which the leader steers the whole traffic towards free flow with a desired velocity.
\begin{remark}
The stabilization results of the Theorems \ref{thm:microconstant} and \ref{thm:microconcave} include the cars $J,\ldots,N-1$.
If we consider the condition \eqref{eq:sufficient} from a different point of view, it designates how large $\ndt$ should be such that our proof provides the stabilization of all cars.
This is achieved by choosing 
\[\ndt\geq\sum_{j=0}^{N-1}\max\{y_i(0),\bar L_i\}.\]
Nevertheless, we note that the numerical results demonstrate that the stabilization of all cars is obtained for every $\ndt>0$.
\end{remark}
The main ingredients to the previous proofs are the maximum principle \eqref{eq:maxmicro} and additionally for the concave kernel that the solutions are monotonicity preserving.
The latter allows us to treat the concave case very similar to the constant case by using the identity \eqref{eq:dyconstant}.
In particular, in both proofs we rely on the relation between the velocity function evaluated at $y_i(t)$ and the equilibrium velocity $\bar v$.
Unfortunately, for non-concave kernels and general initial data such a direct comparison is not possible as we need to treat the nonlinear term in \eqref{eq:dygeneral}. 
Further, the stricter maximum principle \eqref{eq:maxmicro} might be no longer valid.
This makes it challenging to generalize the obtained results.
Nevertheless, the numerical simulations support a similar result for non-concave kernels, see Section \ref{sec:numeric}.\\
The maximum principle provides a sufficient condition to guarantee that a certain number of cars is influenced by the leading vehicle, too.
This avoids the use of the technical set $\BB$ that might be empty. 
If we assume however, that $\BB\neq\emptyset$ holds, we can  generalize the results above:
\begin{corollary}\label{cor:microlyapunovgeneral}
Assume $\BB\neq \emptyset$ with $\BB$ defined by \eqref{eq:bulk}.
Consider a subset $B\subset \BB$ and suppose that either
\begin{enumerate}[label=\alph*)]
  \item the kernel function is constant, i.e. $\wt(x)=\frac{1}{\ndt}$, or
    \item the kernel function is concave, $\lm_i=\lm$ for $i\in B$ and the initial data are monotone increasing (decreasing) in the distances with $y_i(0)\leq\bar L$  ($y_i(0)\geq\bar L$) for $i\in B$.
\end{enumerate}
Then, the stabilization result of the Theorems \ref{thm:microconstant} and \ref{thm:microconcave} hold for a  Lyapunov function  defined as in \eqref{eq:microLyapunov}, where the summation spans the set $B$.
\end{corollary}
\begin{remark}
Note that this includes the results from the Theorems \ref{thm:microconstant} and \ref{thm:microconcave}.
%In particular, by picking the right subsets it is possible to compare the stabilization effect for different values of $\ndt$.
\end{remark}
}

\subsubsection{The macroscopic level}

In the following we consider the macroscopic system \eqref{eq:macrosystem} and obtain similar results as for the microscopic one.
To control and compare the two scales, we need appropriate boundary conditions for system \eqref{eq:macrosystem}, such that the system behaves similarly to the microscopic one.
In particular, in the microscopic case the cars are initially placed on some interval, we have a dynamic for the leading vehicle and no car is entering or leaving the road over time.
We mimic this on the macroscopic level.\\
Consider an initial interval $[a,b]$ with $b>a+\ndt$.
At the left boundary $a$ we prescribe zero inflow conditions such that no cars enter the road as in the microscopic case, i.e.,
\begin{equation*}
(\rho(t,a),\lm(t,a))=(0,\lm_0(a)).
\end{equation*}
As aforementioned we need a corresponding formulation to the leading vehicle at the right boundary.
This vehicle moves with speed $\bar v$ and hence the right boundary should move in time with velocity $\bar v$. 
Therefore, the right boundary is described by $\beta(t)=b+t\bar v$.
Due to the nonlocality of the flux we also need to prescribe the density and Lagrangian marker for the area $[\beta(t),\beta(t)+\ndt]$.
In the microscopic case, no cars are present ahead of the leader such that only the velocity of the leader is taken into account.
To achieve this, we choose constant boundary conditions in such a way that the Lagrangian marker for $y\in[\beta(t),\beta(t)+\ndt]$ is determined by the initial condition at the right boundary, i.e. $\lm(t,y)=\lm_0(b)$.
Then, the equilibrium velocity $\bar v$ together with the Lagrangian marker gives us the equilibrium density for $\rho(t,y)$ and  $y\in[\beta(t),\beta(t)+\ndt]$.
We will denote this density by $\bar \rho_b$.
To sum up, initial values $(\rho_0,\lm_0)$ fulfilling \eqref{eq:initcond} on the interval $[a,b]$ and the boundary conditions
\begin{equation}\label{eq:boundary}
    \begin{cases}
   (\rho(t,a),\lm(t,a))=(0,\lm_0(a))\\
   (\rho(t,y),\lm(t,y))=(\bar \rho_b,\lm_0(b))\quad y\geq \beta(t)=b+t \bar v,
    \end{cases}
\end{equation}
where $\bar \rho_b$ is determined such that $v(\bar \rho_b, \lm_0(b))=\bar v$, are given.
\begin{remark}
Note that we do not study system \eqref{eq:macrosystem} with boundary conditions \eqref{eq:boundary} for its well-posedness.
Nevertheless, the problem \eqref{eq:macrosystem} and \eqref{eq:boundary} on $[a,\beta(t)]$ can also be viewed as an initial value problem on $\R$ by considering the modified initial conditions
\begin{equation}\label{eq:modinitcond}
\tilde \rho_0(x)=\begin{cases}
0,\qquad & \text{ if }x<a,\\
\rho_0(x),\qquad & \text{ if }x\in[a,b],\\
\bar \rho_b,\qquad & \text{ if }x>b
\end{cases},\qquad
\tilde \lm_0(x)=\begin{cases}
\lm_0(a),\qquad & \text{ if }x<a,\\
\lm_0(x),\qquad & \text{ if }x\in[a,b],\\
\lm_0(b),\qquad & \text{ if }x>b.
\end{cases}
\end{equation}
The waves induced by this initial value problem create the same waves as the boundary conditions \eqref{eq:boundary}.
Well-posedness for similar systems as \eqref{eq:modinitcond} is analyzed in \cite{friedrich2020micromacro}.
\end{remark}
After having defined the boundary conditions, which provide comparable results to the microscopic system, we also derive the equilibrium density defined by $\bar \rho(t,x)$.
In general, the equilibrium density $\bar \rho(t,x)$ at $(t,x)$ can be computed, for a given $\lm(t,x)$, by the relationship $v(\bar \rho(t,x),\lm(t,x))=\bar v$.
Furthermore, we observe that $\lm(t,x)$ fulfills a transport equation, see \eqref{eq:macrosystem}.
Hence, $\bar \rho$ satisfies the same partial differential equation, i.e.
\[\partial_t \bar \rho(t,x)+V(t,x) \partial_x \bar \rho(t,x)=0,\]
with initial conditions $\bar \rho_0$, which satisfy $v(\bar \rho_0(x),\lm_0(x))=\bar v$.
This alternative formulation will be useful later on.\\
Now we can turn to an appropriate Lyapunov function and our main result of this section:
\begin{theorem}\label{thm:macroLyapunovDensity}
Let $\rho\in C^1(\R^+;H^2(\R))$ and \rv{Assumption \ref{ass:vprime} hold.
Further, we assume either
\begin{enumerate}[label=\alph*)]
    \item a constant kernel function, or 
    \item a concave kernel function, where for $x\in [b-\ndt,b]$ the initial data shall satisfy that $\lm_0(x)=\lm$, $\rho_0(x)\geq \bar \rho$ ($\rho_0(x)\leq \bar \rho$), $\rho_0$ monotone decreasing (increasing) on $[b-\ndt,b]$.
\end{enumerate}
} We define
\begin{equation}\label{eq:macroLyapunov}
    L(t):=\int_{\alpha(t)}^{\beta(t)}(\rho(t,x)-\bar \rho(t,x))^2 dx,
\end{equation}
with
\begin{equation}
    \label{eq:x0t}
    \alpha(t):=\sup \{x\geq \beta(t)-\ndt : \int_x^{\beta(t)}\rv{\rho(t,y)dy}\geq c_\rho \},
\end{equation}
where $c_\rho:=\min_{t\geq 0} \int_{\beta(t)-\ndt}^{\beta(t)}\rho(t,x)dx$.
Then, we obtain the following bound
\begin{align*}
    L(t)\leq L(0) \exp\left(2\, v'_{\max}\, \rho_{\min} \rv{\int_0^t\wt(\beta(s)-\alpha(s))ds}\right),
\end{align*}
where $\rho_{\min}:=\inf_{x\in [\alpha(0),b]} \min\{\rho_0(x), \bar \rho(0,x)\}$.
\end{theorem}
Before turning to the proof of \rv{Theorem \ref{thm:macroLyapunovDensity}} we want to discuss the lower boundary $\alpha(t)$ given by \eqref{eq:x0t} in the Lyapunov function \eqref{eq:macroLyapunov} in more detail.
To this end we consider the Lyapunov function \eqref{eq:microLyapunov} of \rv{the Theorems \ref{thm:microconstant} and \ref{thm:microconcave}}.
Here, the number of cars is constant.
Hence, similar to the boundary data, a macroscopic equivalent is needed.
This yields a constant mass, i.e. $\int_{\beta(t)-\ndt}^{\beta(t)}\rho(t,x)dx=c$ with $c>0$. 
However, the mass in the interval $[\beta(t)-\ndt,\beta(t)]$ is changing over time and therefore we need to consider a boundary which yields a subinterval $[\alpha(t),\beta(t)]\subset [\beta(t)-\ndt,\beta(t)]$ such that its mass is $c_\rho$.
This \rv{is} exactly achieved by the boundary $\alpha(t)$ as defined in \eqref{eq:x0t}.\\
We can also reformulate the boundary in an intuitive way as seen in the following lemma:
\begin{lemma}\label{lem:x0}
Let $\alpha(t)$ be defined by \eqref{eq:x0t} and assume $\alpha\in C^1(\R^+)$. Then, it is given by the solution of the ordinary differential equation
\begin{equation}\label{eq:x0diff}
    \frac{d}{dt}\alpha(t)=V(t,\alpha(t)),
\end{equation}
with initial condition 
\[ \alpha(0):=\sup \{x\geq b-\ndt : \int_x^{b}\rv{\rho_0(y)dy}\geq c_\rho \}.\]
\end{lemma}
\begin{proof}
First note that due to the continuity of the integral we have $\int_x^{\beta(t)}\rv{\rho(t,y)dy}= c_\rho$ for all $t\geq 0$ \rv{and $x\geq \beta(t)-\ndt$}. Hence, we obtain
\[ \int_{\alpha(t)}^{\beta(t)}\rho(t,x)dx= c_\rho \quad\text{and}\quad \frac{d}{dt}\int_{\alpha(t)}^{\beta(t)}\rho(t,x)dx= 0.\]
On the other hand we have 
\begin{align*}
    \frac{d}{dt}\int_{\alpha(t)}^{\beta(t)}\rho(t,x)dx=& \frac{d}{dt} \beta(t)\rho(t,\beta(t))-\frac{d}{dt} \alpha(t) \rho(t,\alpha(t)) -\int_{\alpha(t)}^{\beta(t)}\partial_x \left(V(t,x) \rho(t,x)\right)dx\\
    =& \bar v\,\rho(t,\beta(t))-\frac{d}{dt} \alpha(t) \rho(t,\alpha(t)) -\bar v\rho(t,\beta(t))+V(t,\alpha(t)) \rho(t,\alpha(t))\\
    =& \left(V(t,\alpha(t))-\frac{d}{dt} \alpha(t)\right) \rho(t,\alpha(t)).
\end{align*}
Since $\alpha(t)$ is defined as the supremum the density $\rho(t,\alpha(t))$ cannot be zero. Hence, the claim follows.
\end{proof}
The meaning of the latter lemma is that $\alpha(t)$ moves with the nonlocal speed $V(t,\alpha(t))$.
In contrast to that, the left boundary $\beta(t)$ moves with the speed $\bar v$ (as the leading vehicle $x_N(t)$ in the microscopic case).
Note that even with this reformulation we need the knowledge of the solution of $\rho$ for the whole time horizon to determine $c_\rho$.
We will comment on specific cases which allow to simplify this assumption.
\begin{proof}[Proof of \rv{Theorem \ref{thm:macroLyapunovDensity}}]
First note that the boundary conditions simplify the nonlocal term $\eqref{eq:NV:conv}$ and its spatial derivative for $x\in [\beta(t)-\ndt,\beta(t)]$.
Here, we have
\rv{
\begin{align*}
    V(t,x)&=\int_x^{\beta(t)}\wt(y-x) (v(\rho(t,y),\lm(t,y))-\bar v)\dy,\\
    \partial_x V(t,x)&= -\int_x^{\beta(t)}\wt'(y-x) (v(\rho(t,y),\lm(t,y))-\bar v)\dy+\wt(0)\left(\bar v- v(\rho(t,x),\lm(t,x)) \right).
\end{align*}}
This derivative will play a key role for proving our stabilization results.\\
Turning now to the Lyapunov function \eqref{eq:macroLyapunov}, we directly obtain 
\begin{align*}
    \frac{d}{dt}L(t)=&\bar v\, (\rho(t,\beta(t))-\bar \rho(t,\beta(t)))^2 - \frac{d}{dt}\alpha(t)(\rho(t,\alpha(t))-\bar \rho(t,\alpha(t)))^2\\
    &+2\int_{\alpha(t)}^{\beta(t)}(\rho(t,x)-\bar \rho(t,x))\frac{d}{dt}(\rho(t,x)-\bar \rho(t,x)) dx.
\end{align*}
Due to the boundary conditions, we have $\rho(t,\beta(t))=\bar \rho(t,\beta(t))$.
This leaves us with
\begin{align*}
    \frac{d}{dt}L(t)=&- \frac{d}{dt}\alpha(t)(\rho(t,\alpha(t))-\bar \rho(t,\alpha(t)))^2\\
    &+2\int_{\alpha(t)}^{\beta(t)}(\rho(t,x)-\bar \rho(t,x))(\partial_t \rho(t,x)-\partial_t\bar \rho(t,x)) dx.
\end{align*}
As $\partial_t \rho(t,x)=-\partial_x (\rho(t,x)V(t,x))$ and $\partial_t \bar \rho(t,x)=-V(t,x)\partial_x \bar \rho(t,x)$, we have
\begin{align*}
    \frac{d}{dt}L(t)=&- \frac{d}{dt}\alpha(t)(\rho(t,\alpha(t))-\bar \rho(t,\alpha(t)))^2\\
    &+2\int_{\alpha(t)}^{\beta(t)}(\rho(t,x)-\bar \rho(t,x))(-V(t,x)) (\partial_x \rho(t,x)-\partial_x \bar\rho(t,x)) dx\\
    &-2\int_{\alpha(t)}^{\beta(t)}(\rho(t,x)-\bar \rho(t,x))\rho(t,x)\partial_x V(t,x) dx.
\end{align*}
For the middle term a partial integration using the boundary conditions yields
\begin{align*}
    \int_{\alpha(t)}^{\beta(t)}-V(t,x) \partial_x (\rho(t,x)-\bar \rho(t,x))^2 dx=&V(t,\alpha(t)) (\rho(t,\alpha(t))-\bar \rho(t,\alpha(t)))^2\\
    &+\int_{\alpha(t)}^{\beta(t)}\partial_x V(t,x)  (\rho(t,x)-\bar \rho(t,x))^2 dx.
  %  =&(V(t,\alpha(t))-\bar v) (\rho(t,\alpha(t))-\bar \rho(t,\alpha(t)))^2+\frac1\ndt\int_{\alpha(t)}^{\beta(t)}\partial_\rho v(\xi(t,x),\lm(t,x))(\bar \rho(t,x)-\rho(t,x))  (\rho(t,x)-\bar \rho(t,x))^2 dx
\end{align*}
\rv{Now, putting everything together, we obtain}
\begin{align}
   \nonumber \frac{d}{dt}L(t)=& \left(V(t,\alpha(t))- \frac{d}{dt}\alpha(t)\right) (\rho(t,\alpha(t))-\bar \rho(t,\alpha(t)))^2 \\
  \label{eq:Ltstart}  &+\int_{\alpha(t)}^{\beta(t)}\partial_x V(t,x)(\bar \rho(t,x)+\rho(t,x))  (\bar\rho(t,x)- \rho(t,x)) dx
\end{align}
 \rv{Due to \eqref{eq:x0diff} the first term is equal to zero. Now consider the cases a) and b) separately.\\
 For a constant kernel the derivative of the nonlocal term simplifies to
 \[\partial_x V(t,x)=\frac{1}{\ndt}\left(\bar v-v(\rho(t,x),\lm(t,x))\right).\]
Hence, we apply the mean value theorem in the first argument of $v$, recall $\bar v=v(\bar \rho(t,x),\lm(t,x))$, and obtain}
\begin{align*}
    \frac{d}{dt}L(t)=& \frac1\ndt\int_{\alpha(t)}^{\beta(t)}\partial_\rho v(\xi(t,x),\lm(t,x)))(\bar \rho(t,x)+\rho(t,x))  (\rho(t,x)-\bar \rho(t,x))^2 dx.
\end{align*}
\rv{This term will be estimated from above.}
For the microscopic system we have the maximum principle \eqref{eq:maxmicro}.
As shown in \cite{friedrich2020micromacro}, the microscopic system can be seen as a semi--discretization of the macroscopic system which converges for $N\to \infty$ towards the macroscopic solution.
Hence, \eqref{eq:maxmicro} gives us a maximum principle for the macroscopic system in Lagrange coordinates, too. This translates to $\rho(t,x)\geq \rho_{\min}$ for $x\geq \alpha(t)$.
Now, we obtain 
\begin{align*}
    \frac{d}{dt}L(t)\leq \frac{2\rho_{\min}\,v_{\max}'}{\ndt}\int_{\alpha(t)}^{\beta(t)}  (\rho(t,x)-\bar \rho(t,x))^2 dx.
\end{align*}
Applying Gr\"onwall's inequality yields the result for the case a).\\
\rv{
In case b) we obtain that also the monotonicity of the solution and $\rho(t,x)\geq \bar \rho$ as well as $\rho(t,x)\leq \bar \rho$ respectively are kept for $x\geq \alpha(t)$.
Now, we consider the case of a monotone decreasing initial density and $\rho_0(x)\geq \bar \rho$.
For simplicity and due to $\lm(t,x)=\lm$ for $x\in[\alpha(t),\beta(t)]$, we drop the dependence of $v$ on $\lm $ in the following.
We rewrite the derivative of the nonlocal term as
\begin{align*}
    \partial_x V(t,x)=& -\int_x^{\beta(t)}\wt'(y-x) (v(\rho(t,y))-v(\rho(t,x))\dy\\&+\left(\wt(0)+\int_x^{\beta(t)}\wt'(y-x) \dy\right)\left(\bar v- v(\rho(t,x)) \right).
    \intertext{Since the monotonicity is preserved, we estimate}
    \partial_x V(t,x)\geq&\,\wt(\beta(t)-x) \left(\bar v- v(\rho(t,x)) \right).
\end{align*}
Further, due to $\rho(t,x)\geq \bar \rho$ we are able to bound \eqref{eq:Ltstart} by
\begin{align*}
    \eqref{eq:Ltstart}&\leq \int_{\alpha(t)}^{\beta(t)}\partial_\rho v(\xi(t,x),\lm(t,x)))\,\wt(\beta(t)-x) (\bar \rho(t,x)+\rho(t,x))  (\rho(t,x)-\bar \rho(t,x))^2 dx\\
    &\leq \wt(\beta(t)-\alpha(t))\,\rho_{\min}\,v_{\max}'\int_{\alpha(t)}^{\beta(t)}  (\rho(t,x)-\bar \rho(t,x))^2 dx.
\end{align*}
For monotone increasing initial data the signs change such that we get the same estimate.
Hence, Grönwall's inequality yields the result for the case b), too.
}
\end{proof}
The exponential rates of the microscopic and the macroscopic Lyapunov function are equal  except for the computation of the densities. 
\begin{remark}[Micro-to-macro convergence] At least formally, the fact that the corresponding Lyapunov function has the same  rate is not  surprising, as for $N\to\infty$ the microscopic Lyapunov function converges to the macroscopic \rv{one}.
Note that in the macroscopic problem, we do not restrict \rv{the initial conditions}.
We discuss formally the limit of \eqref{eq:microLyapunov} in more detail.
Define the piecewise constant function
\begin{equation}\label{eq:microdensity}
\rho^N(t,x):=\sum_{i=0}^{N-1} \frac{1}{Ny_i(t)}\caratt{[x_i(t),x_{i+1}(t))}(x),
\end{equation}
where $\caratt{A}(x)$ is the characteristic function being one for $x\in A$ and zero elsewhere.
In \cite{friedrich2020micromacro}, it is shown that \eqref{eq:microdensity} converges for $N\to\infty$ to a limit function $\rho$ which is actually a weak solution to system \eqref{eq:macrosystem}.
Similarly, the microscopic equilibrium density $1/(N\bar L_i)$ converges to the macroscopic quantity $\bar \rho$.
Furthermore, we rewrite the microscopic Lyapunov function using \eqref{eq:microdensity} (and $\bar \rho^N(t,x)$ defined accordingly):
\begin{align*}
\eqref{eq:microLyapunov}&=\sum_{i=\mathcal{J}}^{N-1}\int_{x_i(t)}^{x_{i+1}(t)}\left( \frac{1}{N y_i(t)}-\frac{1}{N \bar L_i}\right)^2dx
%=\sum_{i=\mathcal{J}}^{N-1}\int_{x_i(t)}^{x_{i+1}(t)}\left(\rho^N(t,x)-\bar\rho^N(t,x)\right)^2dx\\
%&
=\int_{x_{\mathcal{J}}(t)}^{x_{N}(t)}\left(\rho^N(t,x)-\bar\rho^N(t,x)\right)^2dx.
\end{align*}
By the choice of the boundary conditions, the correspondence from $x_N(t)$ to the left boundary $\beta(t)$ is obtained.
For the other boundary, observe that 
\begin{equation}
\label{eq:microdensityconstant}
\frac{d}{dt}\int_{x_{\mathcal{J}}(t)}^{x_N(t)} \rho^N(t,x)dx=\frac{d}{dt}\sum_{i=\mathcal{J}}^{N-1}y_i(t)\frac{1}{Ny_i(t)}=0.
\end{equation} 
Hence, the integral of the microscopic density between two specific cars stays constant over time.
Furthermore, we define
\begin{align*}
c_\rho^N:&=\min_{t\geq 0} \int_{x_{N}(t)-\ndt}^{x_N(t)} \rho^N(t,x)dx
=\int_{x_{N}(t^*)-\ndt}^{x_{\mathcal{J}}(t^*)} \rho^N(t,x)dx+\int_{x_{\mathcal{J}}(t^*)}^{x_N(t^*)} \rho^N(t,x)dx,
\end{align*}
where $t^*$ is the time when the car ${\mathcal{J}}$ is the closest to the border $x_N(t)-\ndt$.
If we now pass to the limit $N\to\infty$, the car ${\mathcal{J}}$ reaches  position $x_N(t)-\ndt$, such that the first term disappears.
Using \eqref{eq:microdensityconstant} we observe
\[c_\rho^N\to  \int_{x_{\mathcal{J}}(t)}^{x_N(t)} \rho^N(t,x)dx\]
for $N\to\infty$ and further by construction also $c_\rho^N\to c_\rho$.
Hence, by comparing  to \eqref{eq:x0t}, we have that $x_{\mathcal{J}}(t)$ converges towards the boundary $\alpha(t)$.
\rv{This demonstrates that the rate in the microscopic case for a concave kernel converges to its macroscopic equivalent as}
\[\rv{\int_0^t \wt(x_N(s)-x_{\mathcal{J}}(s))ds\to \int_0^t\wt(\beta(s)-\alpha(s))ds.}\]
\end{remark}
The value $c_\rho$ depends on the knowledge of the whole solution over time.
Hence, the theoretical boundary $\alpha(t)$ may not be useful in real-world applications.
Nevertheless, in specific cases we \rv{can determine} its value a prior: \\
Consider a case where the density in the interval $[\beta(t)-\ndt, \beta(t)]$ is only increasing over time.
\begin{lemma}\label{lem:rhoL1minsimple}
Let the initial conditions for the Lagrangian marker be $\lm_0(x)=\lm$\rv{, hence $v(\rho,\lm)=:v(\rho)$ and the initial density either satisfies
\begin{enumerate}[label=\alph*)]
    \item  $\rho_0(x)\leq \bar \rho\ \forall x \in[a,b]$, or   
    \item the condition 
\begin{equation}\label{eq:initrhoL1simple}
V(0,b-\ndt)\geq \bar v
\end{equation}
for a constant kernel and a linear velocity function.
\end{enumerate}
}
\noindent Then,
\[c_\rho=\int_{b-\ndt}^{b} \rho_0(x) dx\]
holds.
\end{lemma}
\begin{proof}
Similar to the proof of Lemma \ref{lem:x0} we obtain
\begin{align*}
    \frac{d}{dt}\int_{\beta(t)-\ndt}^{\beta(t)}\rho(t,x)dx=& \left(V(t,\beta(t)-\ndt)-\bar v\right) \rho(t,\beta(t)-\ndt).
\end{align*}
To prove the claim, the derivative needs to be positive.
\rv{The density is already nonnegative.
In case a) we obtain the result as a direct consequence of the strong maximum principle in \cite[Theorem 2.2]{friedrich2018godunov}, which we can apply due to $\lm_0(x)=\lm$.
In particular, this maximum principle holds for all kernels as in \eqref{eq:ass:kernel}.
This gives us $\rho(t,x)\leq \bar \rho$ or equivalently $v(\rho(t,x))\geq \bar v$.
In particular, $V(t,\beta(t)-\ndt)\geq \bar v$ follows.\\
The second case is more involved.
First note that due to Assumption \ref{ass:vprime} and $v$ being linear in $\rho$ we can set $v'_{\max}:=v'(\rho)$.
Now, we study}
\begin{align*}
\frac{d}{dt}\left(V(t,\beta(t)-\ndt)-\bar v\right)&=\frac{1}{\ndt}\left(
\bar v \left[\bar v-v(\rho(t,\beta(t)-\ndt))\right]+v'_{\max}\int_{\beta(t)-\ndt}^{\beta(t)} \partial_t \rho(t,y) dy \right)\\
&=\frac{1}{\ndt}\left(
\bar v\, v'_{\max}\left[\bar \rho -\rho(t,\beta(t)-\ndt)\right]-v'_{\max}\int_{\beta(t)-\ndt}^{\beta(t)} \partial_x \left(\rho(t,y)V(t,y)\right) dy \right)\\
%&=\frac{1}{\ndt}\left(
%\bar v v'_{\max}\left[\bar \rho -\rho(t,\beta(t)-\ndt)\right]-v'_{\max}\left[ \bar v \bar \rho - \rho(t,\beta(t)-\ndt)V(t,\beta(t)-\ndt)\right] \right)\\
&=\frac{1}{\ndt}\rho(t,\beta(t)-\ndt)v'_{\max}\left[ V(t,\beta(t)-\ndt)-\bar v\right].
\end{align*}
This gives us an ordinary differential equations \rv{whose} solution is given by
\begin{align*}
\left(V(t,\beta(t)-\ndt)-\bar v\right)&=\left(V(0,b-\ndt)-\bar v\right)\exp\left(\frac{v'_{\max}}{\ndt}\int_0^t\rho(s,\beta(s)-\ndt)ds\right).
\end{align*}
Hence, if we assume \eqref{eq:initrhoL1simple}, it follows that $V(t,\beta(t)-\ndt)\geq \bar v$ for $t\geq 0$ and therefore also $\frac{d}{dt}\int_{\beta(t)-\ndt}^{\beta(t)}\rho(t,x)dx\geq 0$.
This yields the desired result.
\end{proof}
\rv{Note that the condition a) prohibits to use Lemma \ref{lem:rhoL1minsimple} for decreasing initial data and a concave kernel function in Theorem \ref{thm:macroLyapunovDensity}.}\\
Alternatively we may find a lower bound on $c_\rho$ that is determined a priori and then use a modified left boundary as outlined in the following remark.
\begin{remark}
Let us again consider a first order model, i.e. $\lm_0(x)=\lm$.
If we view the problem as an initial value problem on $\R$, the initial conditions are given by 
\begin{equation}\label{eq:modinitcond2}
\tilde \rho_0(x)=\begin{cases}
\bar \rho_a,\qquad & \text{ if }x<a,\\
\rho_0(x),\qquad & \text{ if }x\in[a,b],\\
\bar \rho_b,\qquad & \text{ if }x>b
\end{cases},\qquad
\tilde \lm_0(x)=\begin{cases}
\lm_0(a),\qquad & \text{ if }x<a,\\
\lm_0(x),\qquad & \text{ if }x\in[a,b],\\
\lm_0(b),\qquad & \text{ if }x>b,
\end{cases}
\end{equation}
where $\bar \rho_a$ is determined such that $v(\bar \rho_a, \lm_0(a))=\bar v$. Furthermore, the initial conditions must be chosen in a way such that $\bar \rho_a,\bar \rho_b, \rho_0(x)>0$ holds.
Then, by the maximum principle \cite{friedrich2018godunov} we have $\rho(t,x)\geq \min\{ \bar \rho_a,\bar \rho_b, \inf_{x\in[a,b]}\rho_0(x)\}$.
This allows to bound $c_\rho$ by $\ndt\, \min\{ \bar \rho_a,\bar \rho_b, \inf_{x\in[a,b]}\rho_0(x)\}$ from below.
Finally, this can be used for a left boundary $\tilde \alpha(t)\geq \alpha(t)$ which depends on the lower bound.
Here, the same results as in \rv{Theorem \ref{thm:macroLyapunovDensity}} are obtained by simply replacing the lower bound.
\rv{In contrast to Lemma \ref{lem:rhoL1minsimple}, there are no further assumptions on the initial data or velocity function.}
Note that due to the different inflow the solution does not correspond to the microscopic solution.
Further, note that such a strategy is not possible for the second order model, since the maximum principle does not provide a lower bound which is strictly positive.
\end{remark}
Let us close this section by comparing our results to other approaches considered in the literature.
As already mentioned in the introduction, \cite{huang2020stability} considers the stability of a first order model with linear velocity on a ring road. Our model coincides with this model for a constant Lagrangian marker and the corresponding velocity function. One of the main results in \cite{huang2020stability} is that for a constant kernel, specific initial data and a specific nonlocal reach $\ndt$ traveling waves can be created, such that no control is possible.
In particular, this is the reason why in \cite{karafyllis2022control} a different model is considered.
Nevertheless, this is \rv{not in contradiction with} our results as we do not consider a ring road and such phenomena may not occur.

\section{Numerical simulations}\label{sec:numeric}
\subsection{Numerical schemes}
To illustrate the theoretical results we  use numerical schemes to solve the microscopic equations \eqref{eq:micronewnotation} and the macroscopic system \eqref{eq:macroconservative}.\\
To compare the microscopic and macroscopic approaches we will choose the same initial density and Lagrangian marker for both scales.
To construct the initial placement of the microscopic cars for a given initial density and number of cars we use the approach outlined in \cite{DiFrancescoFagioliRosini}.
The microscopic equations \eqref{eq:micronewnotation} are a system of ordinary differential equations and solved by the method \texttt{ode23s} of \texttt{MATLAB}.\\
For the macroscopic equations we discretize  system \eqref{eq:macroconservative} using  a numerical scheme presented in \cite{friedrich2020micromacro, friedrich2018godunov}  similar to the Roe scheme presented in \cite{roe1981approximate}.
The two equations in \eqref{eq:macroconservative} are only coupled via the nonlocal term, that is computed at each time step. 
% Therefore, the two equations can be approximated separately.
This leads to the following numerical flux
\[F_{j+\frac12}^n=V_j^n \begin{pmatrix} \rho_j^n\\q_j^n\end{pmatrix},\]
where
\[V_j^n =\sum_{k=0}^{\lfloor\ndt/\Delta x\rfloor-1}\gamma_k V\left(\rho^n_{j+k+1},\frac{q_{j+k+1}^n}{\rho_{j+k+1}^n}\right)\quad\text{ and }\quad\gamma_k=\int_{k\Delta x}^{(k+1)\Delta x} \wt(y-x)dy.\]
The full scheme is then given by
\begin{equation}\label{eq:macroscheme}
     \begin{pmatrix} \rho_j^{n+1}\\q_j^{n+1}\end{pmatrix}=\begin{pmatrix} \rho_j^{n}\\q_j^{n}\end{pmatrix}-\lambda \left(F^n_{j+\frac12}-F^n_{j-\frac12}\right).
\end{equation}
At the boundaries the fluxes are given by \eqref{eq:boundary}.
The CFL condition is chosen using an adaptive step size control determined by the maximal nonlocal velocity, i.e. $\Delta t^n\leq \Delta x/\max_j V_j^n $. 

\subsection{Numerical example}
We consider the starting interval $[-1.5,1.5]$. In the microscopic case $N=500$ cars \rv{are placed in the interval according to the initial density $\rho_0(x)$} and in the macroscopic case we have \rv{$a=-1.5$ and $b=1.5$}. The spatial step size is given by $\Delta x=2.5\cdot 10^{-3}$. The constant $\ndt$ is $\ndt=0.5$ and the initial conditions are given by
\begin{align*}
\rho_0(x)=\begin{cases}
0.5,&\text{if }x<0,\\
0.3,&\text{if }x>0,
\end{cases}
\qquad\text{and}\qquad
\lm_0(x)=\rv{\begin{cases}
1,&\text{if }x<0,\\
\frac{5}{8},&\text{if }x>0.
\end{cases}}
\end{align*}
The equilibrium velocity of the right boundary and the leading vehicle respectively is given by \rv{$\bar v=0.5$}.
The velocity function is chosen as \rv{$v(\lm,\rho)=\lm(1-\rho)$}.
At $t=0$ the system  is already at its stationary state is for data with $x<0.$
It will be driven away from the equilibrium due to the traffic ahead.
Furthermore, the initial densities are larger or equal to the equilibrium density.
Hence, in the microscopic case the distance between all cars in the set $\BB$ and its follower is always smaller or equal to the equilibrium distance due to the maximum principle \eqref{eq:maxmicro}.
In particular, this implies that $J=\mathcal{J}$.
\rv{In addition, the initial conditions are chosen such that the assumptions of Theorems \ref{thm:microconstant}, \ref{thm:microconcave} and \ref{thm:macroLyapunovDensity} are fulfilled.\\
We consider first a constant kernel $\wt^{\text{const.}}=1/\ndt$.}
Figure \ref{fig:timesnapshot} shows the density in the microscopic and macroscopic case at different times.
Note that due to the applied boundary conditions, we have zero inflow at the left boundary of the macroscopic solution such that the density drops to zero. 
The theoretical results  demonstrate only the stabilization of the density in the interval ahead of the red dotted lines.
Nevertheless, also the density outside this area seems to converge towards the steady state solution.
\rv{
\begin{figure}
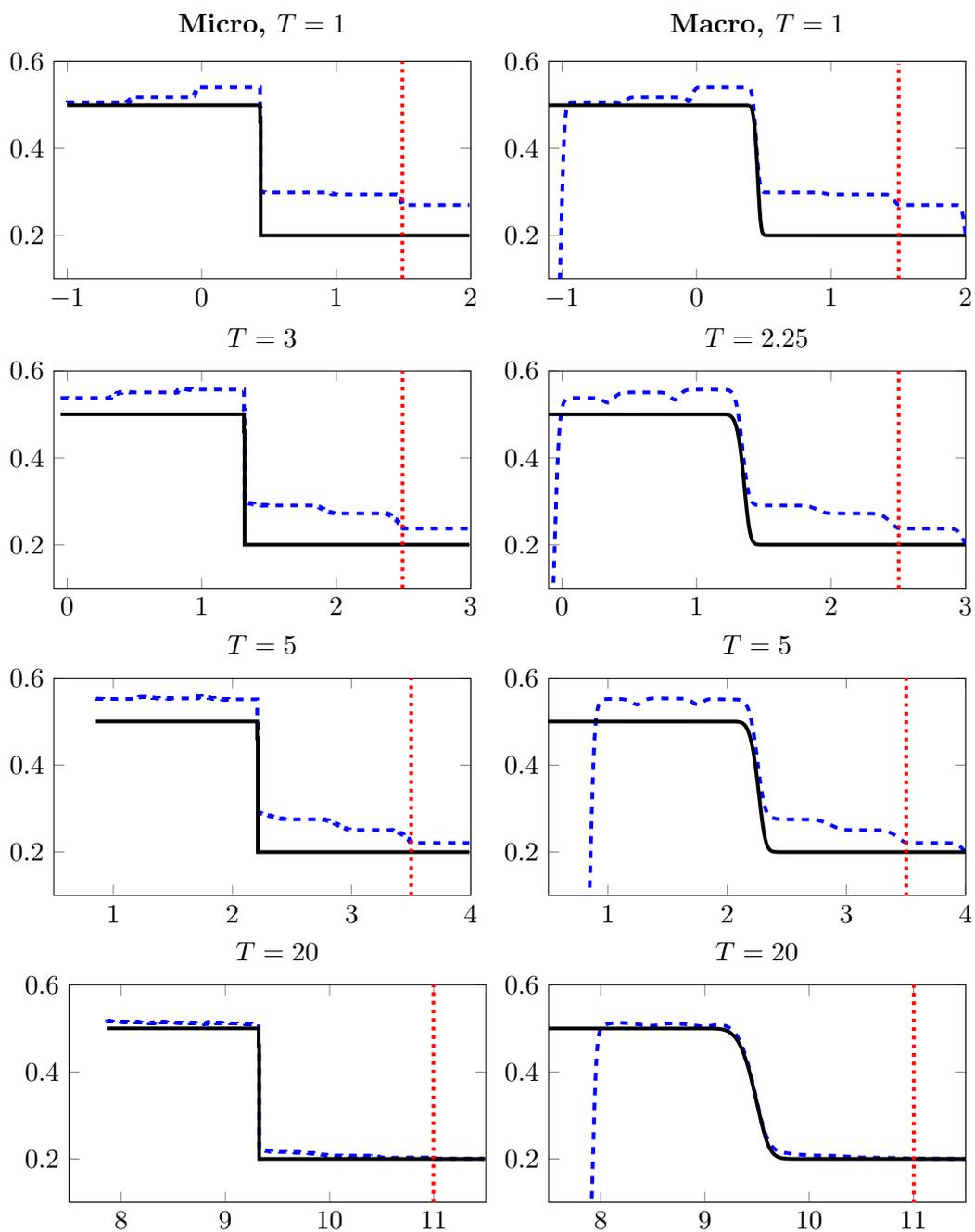

\setlength{\fwidth}{0.4\textwidth}
\begin{center}
\input{Plotsnew/microdensityT1}
\begin{tikzpicture}

\begin{axis}[%
width=0.951\fwidth,
height=0.5\fwidth,
at={(0\fwidth,0\fwidth)},
scale only axis,
xmin=-1.1,
xmax=2,
ymin=0.1,
ymax=0.6,
axis background/.style={fill=white},
title style={font=\bfseries},
title={Macro, $T=1$}
]
\addplot [color=blue, dashed, line width=1.5pt, forget plot]
  table[row sep=crcr]{%
-1.49875	2.22044604925031e-16\\
-1.05625	0.000126176006980838\\
-1.05125	0.000424858984864329\\
-1.04625	0.00128095840988118\\
-1.04375	0.00213455289439546\\
-1.04125	0.00346127428862197\\
-1.03875	0.00546245041032956\\
-1.03625	0.00839150356481366\\
-1.03375	0.0125513196254736\\
-1.03125	0.0182833932572406\\
-1.02875	0.0259473345063215\\
-1.02625	0.0358909272492143\\
-1.02375	0.048413081725224\\
-1.02125	0.0637242371143409\\
-1.01875	0.081910375470547\\
-1.01625	0.102907169993956\\
-1.01375	0.126489579095871\\
-1.00875	0.179771359639131\\
-0.99625	0.321391817008813\\
-0.99125	0.36987900085542\\
-0.98875	0.390979661747061\\
-0.98625	0.409829697235357\\
-0.98375	0.42642443960992\\
-0.98125	0.440830089382476\\
-0.97875	0.453167054701173\\
-0.97625	0.463593992366525\\
-0.97375	0.472293464591466\\
-0.97125	0.479459706933229\\
-0.96875	0.48528868859826\\
-0.96625	0.489970430804167\\
-0.96375	0.493683414117472\\
-0.96125	0.496590830704889\\
-0.95875	0.498838403480926\\
-0.95625	0.500553486795559\\
-0.95375	0.501845172675127\\
-0.95125	0.502805146440748\\
-0.94625	0.504018237538653\\
-0.94125	0.504637088114381\\
-0.93375	0.505017372660738\\
-0.91875	0.505159466675437\\
-0.74875	0.505167703055856\\
-0.59375	0.504989972550911\\
-0.57875	0.504458765171987\\
-0.55625	0.503409895826442\\
-0.54875	0.50370398931652\\
-0.54375	0.504265164791942\\
-0.53625	0.505678969792654\\
-0.52875	0.507641106233488\\
-0.50875	0.513262904551089\\
-0.50125	0.514777141421347\\
-0.49375	0.515823207855489\\
-0.48625	0.516466311980159\\
-0.47625	0.516890549690441\\
-0.45875	0.51708505923649\\
-0.36375	0.517104734403123\\
-0.10875	0.516928617640943\\
-0.0987499999999999	0.516480830308931\\
-0.0912500000000001	0.515715073181453\\
-0.08375	0.51439828226336\\
-0.0762499999999999	0.512538749311521\\
-0.0662499999999999	0.509947094974672\\
-0.06125	0.509174115353063\\
-0.0562499999999999	0.50915363438362\\
-0.05375	0.509502018866996\\
-0.05125	0.510113409107319\\
-0.0487500000000001	0.510991839183613\\
-0.04375	0.513505322441925\\
-0.0387500000000001	0.516842619884296\\
-0.0287500000000001	0.524612922128637\\
-0.0237499999999999	0.528329228687922\\
-0.01875	0.531580507442928\\
-0.0137499999999999	0.534243801606289\\
-0.00875000000000004	0.536299484836099\\
-0.00374999999999992	0.537800995827051\\
0.00124999999999997	0.538841646883857\\
0.00875000000000004	0.539768004304795\\
0.01875	0.540294625892499\\
0.0337499999999999	0.540484691193933\\
0.11375	0.540507428934051\\
0.37625	0.540305608518017\\
0.38375	0.539880455566108\\
0.38875	0.539253722263348\\
0.39375	0.53812067918793\\
0.39625	0.537273408159626\\
0.39875	0.536177256885464\\
0.40125	0.534777494749395\\
0.40375	0.53301310704266\\
0.40625	0.530817752148707\\
0.40875	0.528121290510052\\
0.41125	0.524851937470202\\
0.41375	0.520939056672338\\
0.41625	0.516316563040172\\
0.41875	0.510926844945992\\
0.42125	0.504725046040581\\
0.42375	0.497683472668494\\
0.42625	0.489795819958167\\
0.43125	0.471585110226155\\
0.43625	0.45058301136759\\
0.44375	0.415992702816477\\
0.45125	0.381691060194867\\
0.45625	0.361201943864337\\
0.46125	0.343742886497687\\
0.46375	0.336305767867939\\
0.46625	0.329751774469028\\
0.46875	0.324062555970957\\
0.47125	0.31919726053628\\
0.47375	0.315097669155036\\
0.47625	0.311693516029332\\
0.47875	0.308907568295183\\
0.48125	0.306660144842118\\
0.48375	0.304872869890353\\
0.48625	0.303471566331636\\
0.48875	0.302388286014747\\
0.49375	0.300941867551335\\
0.49875	0.300145504798811\\
0.50625	0.299609881286066\\
0.51875	0.299388962188945\\
0.57625	0.2993605403336\\
0.91125	0.299175349619715\\
0.92875	0.298635077051852\\
0.94375	0.297718007343637\\
0.97875	0.295272776345212\\
0.99375	0.294844252185843\\
1.01875	0.294735857012042\\
1.41125	0.294421163753756\\
1.42125	0.29389725835311\\
1.42875	0.293150277955797\\
1.43625	0.291967614795867\\
1.44375	0.290251636542957\\
1.45125	0.28796204363044\\
1.45875	0.285146672010886\\
1.47125	0.27974567845628\\
1.48125	0.27553835482735\\
1.48875	0.272965803598792\\
1.49375	0.27169437938113\\
1.49875	0.270856340262391\\
1.50625	0.27032123509425\\
1.51875	0.270128138931173\\
1.59125	0.270106474101496\\
1.93625	0.269889242187154\\
1.94375	0.269483184548725\\
1.94875	0.268918206789238\\
1.95375	0.267941634533697\\
1.95875	0.266335738597166\\
1.96125	0.265212435170408\\
1.96375	0.263824475709835\\
1.96625	0.262131356766154\\
1.96875	0.260092548617382\\
1.97125	0.257669384029293\\
1.97375	0.254827281145678\\
1.97625	0.251538171115534\\
1.97875	0.247782817902858\\
1.98125	0.243552482748157\\
1.98625	0.233690089978507\\
1.99125	0.222165583169065\\
1.99875	0.202789273597503\\
};
\addplot [color=black, line width=1.5pt, forget plot]
  table[row sep=crcr]{%
-1.49875	0.5\\
0.37875	0.499782259509143\\
0.38625	0.499361048086364\\
0.39125	0.498770079802809\\
0.39625	0.497744722415542\\
0.40125	0.496053151622091\\
0.40375	0.494865770976403\\
0.40625	0.493392835041679\\
0.40875	0.491585692263907\\
0.41125	0.489391824450796\\
0.41375	0.486755264964398\\
0.41625	0.483617118353366\\
0.41875	0.479916195215369\\
0.42125	0.475589796012906\\
0.42375	0.470574703399213\\
0.42625	0.464808472325227\\
0.42875	0.458231137232225\\
0.43125	0.450787480242866\\
0.43375	0.442430014664513\\
0.43625	0.433122821697775\\
0.43875	0.42284631844585\\
0.44375	0.399423305109769\\
0.44875	0.372557477927532\\
0.45625	0.328238081741053\\
0.46375	0.284713405929236\\
0.46875	0.259641830696511\\
0.47125	0.248888261112877\\
0.47375	0.239451862722625\\
0.47625	0.231346934238825\\
0.47875	0.224529402033908\\
0.48125	0.218909154975254\\
0.48375	0.214364623014951\\
0.48625	0.21075719787946\\
0.48875	0.207943741910233\\
0.49125	0.205786232870178\\
0.49375	0.204158309309617\\
0.49625	0.202948987300203\\
0.49875	0.202064093381261\\
0.50375	0.200972550858983\\
0.50875	0.200435163898053\\
0.51625	0.200118208641816\\
0.53375	0.200003568841585\\
0.91125	0.2\\
1.99875	0.2\\
};
\addplot [color=red, dotted, line width=1.5pt, forget plot]
  table[row sep=crcr]{%
1.5025	2.22044604925031e-16\\
1.5025	0.594558171827486\\
};
\end{axis}

\end{tikzpicture}%
\input{Plotsnew/microdensityT3}
\begin{tikzpicture}

\begin{axis}[%
width=0.951\fwidth,
height=0.5\fwidth,
at={(0\fwidth,0\fwidth)},
scale only axis,
xmin=-0.1,
xmax=3,
ymin=0.1,
ymax=0.6,
axis background/.style={fill=white},
title style={font=\bfseries},
title={$T=2.25$}
]
\addplot [color=blue, dashed, line width=1.5pt, forget plot]
  table[row sep=crcr]{%
-1.49875	0\\
-0.10875	0.000173051432796445\\
-0.10375	0.000496347367661798\\
-0.0987499999999999	0.00130701725420757\\
-0.0962499999999999	0.00205407119163548\\
-0.09375	0.0031599858468816\\
-0.0912500000000001	0.00475890964152503\\
-0.0887500000000001	0.00701632497999816\\
-0.0862500000000002	0.0101281996436593\\
-0.0837500000000002	0.014316349684647\\
-0.0812500000000003	0.0198192732636091\\
-0.0787500000000003	0.0268782689362794\\
-0.0762499999999999	0.0357194411385318\\
-0.07375	0.04653309580423\\
-0.07125	0.0594528410951942\\
-0.0687500000000001	0.0745372033561487\\
-0.0662500000000001	0.0917565592401313\\
-0.0612500000000002	0.132016450204841\\
-0.0562499999999999	0.178235740766576\\
-0.0412500000000002	0.322941141387768\\
-0.0362499999999999	0.364572381532401\\
-0.03125	0.40054444460645\\
-0.0262500000000001	0.430646398994475\\
-0.0237500000000002	0.443580733029942\\
-0.0212500000000002	0.455194025559835\\
-0.0187500000000003	0.465570488874623\\
-0.0162500000000003	0.474801382311555\\
-0.0137499999999999	0.48298093544288\\
-0.01125	0.490203196844815\\
-0.00875000000000004	0.496559698325587\\
-0.00625000000000009	0.50213780892368\\
-0.00375000000000014	0.507019652470513\\
-0.0012500000000002	0.511281471297556\\
0.0012500000000002	0.514993332430005\\
0.00375000000000014	0.518219088319645\\
0.00625000000000009	0.521016519864272\\
0.00875000000000004	0.523437603983117\\
0.0137499999999999	0.527331746147143\\
0.0187500000000003	0.530215383376239\\
0.0237500000000002	0.532334514716223\\
0.0287500000000001	0.533879786485424\\
0.0337499999999999	0.534997583584202\\
0.0412500000000002	0.536109025370342\\
0.05125	0.536923911525536\\
0.0637500000000002	0.537376785858321\\
0.0862500000000002	0.53759978050878\\
0.17625	0.537633660556911\\
0.24125	0.537399682431758\\
0.25875	0.536875060421675\\
0.27125	0.536076237854734\\
0.28375	0.534753881395118\\
0.29625	0.532859143867748\\
0.32625	0.527702240155147\\
0.33375	0.527019826533125\\
0.34125	0.526876649999577\\
0.34875	0.527357646019083\\
0.35625	0.528475576456013\\
0.36375	0.530166097738107\\
0.37375	0.533079535784205\\
0.40375	0.542521183042872\\
0.41375	0.544904036074034\\
0.42375	0.54672698697478\\
0.43375	0.548039503240424\\
0.44625	0.549106491413152\\
0.46125	0.54980595866618\\
0.48125	0.550194501039388\\
0.51875	0.550343382637211\\
0.72125	0.550269936476472\\
0.74375	0.549892598824056\\
0.75625	0.549304491210453\\
0.76625	0.548461467233904\\
0.77625	0.547148945811376\\
0.78625	0.545269721885674\\
0.79625	0.542822323991204\\
0.82625	0.534752249430135\\
0.83125	0.533968031887128\\
0.83625	0.533532149277881\\
0.84125	0.533491249494908\\
0.84625	0.533868523351359\\
0.85125	0.534660092996861\\
0.85625	0.535834384348924\\
0.86375	0.538184039861706\\
0.87375	0.541986938298425\\
0.88875	0.547731983122383\\
0.89625	0.550146188630504\\
0.90375	0.552113701671302\\
0.91125	0.55362867454238\\
0.91875	0.554737459574005\\
0.92875	0.555710289523803\\
0.94125	0.556366373317028\\
0.95875	0.556726877195246\\
0.99625	0.556851396920227\\
1.20375	0.556731718263302\\
1.22125	0.556356691633446\\
1.23375	0.555637561384574\\
1.24375	0.554508712221342\\
1.25125	0.553140624687487\\
1.25625	0.551885062160312\\
1.26125	0.550282611133637\\
1.26625	0.548262527560148\\
1.27125	0.545746719308218\\
1.27625	0.542650902498377\\
1.28125	0.538886450843461\\
1.28625	0.534363079475719\\
1.29125	0.528992499834919\\
1.29625	0.522693149658677\\
1.30125	0.515396025805443\\
1.30625	0.507051512084141\\
1.31125	0.497636891249748\\
1.31625	0.487163968087131\\
1.32125	0.475685943410852\\
1.32875	0.456815502566025\\
1.33625	0.436456996092578\\
1.35625	0.380979126226222\\
1.36375	0.362215901564227\\
1.36875	0.35090855211504\\
1.37375	0.340712612383908\\
1.37875	0.331693258322472\\
1.38375	0.323861842207191\\
1.38875	0.317183051360638\\
1.39375	0.311584769377349\\
1.39875	0.30696903573207\\
1.40375	0.303222726130624\\
1.40875	0.300226959133443\\
1.41375	0.297864655615846\\
1.41875	0.296026051580958\\
1.42375	0.294612247367949\\
1.43125	0.293103129323612\\
1.43875	0.292123366263996\\
1.44875	0.291347051322989\\
1.46125	0.290863470233515\\
1.48375	0.290572573832947\\
1.54375	0.290509824760436\\
1.77625	0.290296321472955\\
1.80125	0.289800922845973\\
1.82125	0.288952387060551\\
1.83875	0.287745989016451\\
1.85625	0.286060819167365\\
1.87625	0.283619889269441\\
1.94125	0.275124687770257\\
1.95875	0.273624411154713\\
1.97625	0.272669221291433\\
1.99375	0.272202480406023\\
2.02125	0.272026817794198\\
2.20375	0.272011541236831\\
2.30625	0.271788496202454\\
2.32875	0.271304552872697\\
2.34625	0.270488547676581\\
2.36125	0.269305851480492\\
2.37375	0.26787386119023\\
2.38625	0.265965946293997\\
2.39875	0.263544883670983\\
2.41125	0.260614286107467\\
2.42375	0.257230850415716\\
2.44125	0.251974276521593\\
2.46375	0.245184434766704\\
2.47625	0.24195227974492\\
2.48625	0.239913813918679\\
2.49375	0.238802098808923\\
2.50125	0.23810191453023\\
2.51125	0.237652802749103\\
2.52875	0.237410795162709\\
2.58625	0.237357488522795\\
2.88125	0.237147393898314\\
2.89875	0.236703697613994\\
2.91125	0.236000433390729\\
2.92125	0.23503633291968\\
2.93125	0.233548983678614\\
2.93875	0.231983088930192\\
2.94625	0.229945120658613\\
2.95375	0.227364411020907\\
2.96125	0.224186992848997\\
2.96875	0.220387043015104\\
2.97625	0.215978501163204\\
2.98375	0.211024853395434\\
2.99375	0.203784605809717\\
2.99875	0.200018761162031\\
};
\addplot [color=black, line width=1.5pt, forget plot]
  table[row sep=crcr]{%
-1.49875	0.5\\
1.21125	0.499767714869196\\
1.22625	0.499285688541607\\
1.23625	0.498589649558855\\
1.24375	0.497731402125924\\
1.25125	0.496455010411583\\
1.25875	0.494613136711898\\
1.26375	0.492985715150718\\
1.26875	0.490971483578669\\
1.27375	0.488507777892153\\
1.27875	0.485528215868334\\
1.28375	0.481963570090874\\
1.28875	0.477742794049399\\
1.29375	0.472794227128303\\
1.29875	0.467047038488319\\
1.30375	0.460433013074522\\
1.30875	0.45288883119766\\
1.31375	0.44435903859085\\
1.31875	0.434799932756397\\
1.32375	0.42418458127951\\
1.32875	0.412509105596322\\
1.33375	0.39980016799924\\
1.33875	0.386123251048149\\
1.34625	0.364053397334268\\
1.35625	0.332742594283722\\
1.36875	0.293732483766048\\
1.37375	0.279217786720582\\
1.37875	0.265784129861496\\
1.38375	0.25366111102643\\
1.38875	0.242996004142073\\
1.39375	0.233846920857144\\
1.39875	0.226188292440744\\
1.40375	0.219926087870421\\
1.40875	0.214918417529871\\
1.41375	0.210996993799529\\
1.41875	0.207985968869158\\
1.42375	0.205716237859613\\
1.42875	0.204034747652798\\
1.43375	0.202809347629871\\
1.44125	0.201592066481518\\
1.44875	0.200876147473081\\
1.45875	0.200377746234921\\
1.47625	0.20007663320263\\
1.52125	0.200000613738288\\
2.99875	0.2\\
};
\addplot [color=red, dotted, line width=1.5pt, forget plot]
  table[row sep=crcr]{%
2.5025	0\\
2.5025	0.612554213014179\\
};
\end{axis}
\end{tikzpicture}%
\input{Plotsnew/microdensityT5}
\begin{tikzpicture}

\begin{axis}[%
width=0.951\fwidth,
height=0.5\fwidth,
at={(0\fwidth,0\fwidth)},
scale only axis,
xmin=0.5,
xmax=4,
ymin=0.1,
ymax=0.6,
axis background/.style={fill=white},
title style={font=\bfseries},
title={$T=5$}
]
\addplot [color=blue, dashed, line width=1.5pt, forget plot]
  table[row sep=crcr]{%
-1.49875	0\\
0.80125	0.000174059672231941\\
0.80625	0.000505036413464222\\
0.81125	0.00134258776734519\\
0.81375	0.00211845140214217\\
0.81625	0.00327056599513487\\
0.81875	0.004940583857171\\
0.82125	0.00730326670196257\\
0.82375	0.0105653933847645\\
0.82625	0.0149605384786775\\
0.82875	0.0207389486438645\\
0.83125	0.0281523700865955\\
0.83375	0.0374345593558316\\
0.83625	0.0487791954572958\\
0.83875	0.0623177767172676\\
0.84125	0.0781005770687981\\
0.84375	0.0960836547860211\\
0.84875	0.137985184406944\\
0.85375	0.185838509050023\\
0.86625	0.310910295500637\\
0.87125	0.355705892322594\\
0.87625	0.39494570080064\\
0.88125	0.428144498641594\\
0.88375	0.442510798633564\\
0.88625	0.455463635751703\\
0.88875	0.467082556428948\\
0.89125	0.477458007117987\\
0.89375	0.486686279961254\\
0.89625	0.494865513909505\\
0.89875	0.502092655868883\\
0.90125	0.508461254154102\\
0.90375	0.514059945905028\\
0.90625	0.518971503774138\\
0.90875	0.523272319254499\\
0.91125	0.527032216209212\\
0.91375	0.53031450557591\\
0.91625	0.533176209036503\\
0.92125	0.537836580568688\\
0.92625	0.54135792109007\\
0.93125	0.544010205507108\\
0.93625	0.546002581216433\\
0.94125	0.547495731538475\\
0.94875	0.549059537241784\\
0.95625	0.550066271557364\\
0.96625	0.550869546783107\\
0.97875	0.551383297725876\\
0.99875	0.55169602933462\\
1.04625	0.551804138852639\\
1.10875	0.551575707156466\\
1.13125	0.551064990237167\\
1.14875	0.550195999350032\\
1.16375	0.548948936969988\\
1.17875	0.547157029699999\\
1.19625	0.544487393732222\\
1.21875	0.540981139953948\\
1.22875	0.539840191013733\\
1.23875	0.539199493231168\\
1.24625	0.539119651778773\\
1.25375	0.539407079288005\\
1.26375	0.540337935328627\\
1.27375	0.541791050339782\\
1.28875	0.544550405521639\\
1.30875	0.548250920578143\\
1.32125	0.55009483258059\\
1.33375	0.551440527266022\\
1.34625	0.55232068041972\\
1.36125	0.552903384533675\\
1.38125	0.553190077314119\\
1.42375	0.553183239197704\\
1.63375	0.552537361987389\\
1.65125	0.551811295820884\\
1.66625	0.550706103069851\\
1.67875	0.549339005830375\\
1.69125	0.547538070836145\\
1.70875	0.544462449807869\\
1.72625	0.541388450360769\\
1.73625	0.540052224913779\\
1.74375	0.53941556047323\\
1.75125	0.539168143940298\\
1.75875	0.539343445431273\\
1.76625	0.539934236026458\\
1.77625	0.541277957576287\\
1.78875	0.543542048481706\\
1.81125	0.547731532274928\\
1.82375	0.549464444346785\\
1.83625	0.550614054600199\\
1.84875	0.551254169740653\\
1.86375	0.551549212503921\\
1.88875	0.551477082716302\\
1.97125	0.550943632215055\\
2.08125	0.550569262117742\\
2.09875	0.55002391847986\\
2.11125	0.549252843913959\\
2.12125	0.548262092705081\\
2.13125	0.546793388695684\\
2.13875	0.545276884576408\\
2.14625	0.543311196537284\\
2.15375	0.540800237889145\\
2.16125	0.537637752243108\\
2.16625	0.535111050921466\\
2.17125	0.53220825206207\\
2.17625	0.528893028145649\\
2.18125	0.525128988854972\\
2.18625	0.520880495209819\\
2.19125	0.516113676883433\\
2.19625	0.510797675812305\\
2.20125	0.504906128493742\\
2.20625	0.498418880905967\\
2.21125	0.49132390266433\\
2.21625	0.483619330877143\\
2.22375	0.470945981983584\\
2.23125	0.457024028066213\\
2.23875	0.442031692772118\\
2.24875	0.420842771276912\\
2.27125	0.372445714629472\\
2.27875	0.357531258332641\\
2.28625	0.343814350683717\\
2.29125	0.335454177574731\\
2.29625	0.32777511591238\\
2.30125	0.320799688519502\\
2.30625	0.314530893436542\\
2.31125	0.308954407386718\\
2.31625	0.304041440326812\\
2.32125	0.299751930458277\\
2.32625	0.296037797156735\\
2.33125	0.292846021891924\\
2.33625	0.290121391354171\\
2.34125	0.287808801733243\\
2.34875	0.284997047465482\\
2.35625	0.28282882479796\\
2.36375	0.281161715091666\\
2.37375	0.279517981169956\\
2.38375	0.27834319811455\\
2.39625	0.277309181592719\\
2.41375	0.27635823637289\\
2.43625	0.275624663022835\\
2.46625	0.275140351945393\\
2.51125	0.274951777625357\\
2.70125	0.274547457269673\\
2.73125	0.273928754599961\\
2.75625	0.272973923791991\\
2.77875	0.271663373678339\\
2.80125	0.26987407582265\\
2.82375	0.26761905008076\\
2.85125	0.264356796916728\\
2.92375	0.255388835883777\\
2.94375	0.253523519405087\\
2.96375	0.252155477558148\\
2.98375	0.251311888906783\\
3.00625	0.250887695684586\\
3.04625	0.250740578172182\\
3.26375	0.25037123092404\\
3.29375	0.249751097947089\\
3.31625	0.24885689107619\\
3.33625	0.247610771187843\\
3.35375	0.24608019749768\\
3.37125	0.244074060921866\\
3.38875	0.241559920642338\\
3.40625	0.238554881404889\\
3.42625	0.234640562160028\\
3.46875	0.226145788662544\\
3.48125	0.224165227026241\\
3.49375	0.222690258970098\\
3.50375	0.221959610474739\\
3.51875	0.22142386932686\\
3.54125	0.221168154396461\\
3.60625	0.221102827120861\\
3.85625	0.220870703244438\\
3.88125	0.220387176430111\\
3.89875	0.21965878902467\\
3.91375	0.21860977400686\\
3.92625	0.217322927761013\\
3.93875	0.215569756817214\\
3.95125	0.213282801585712\\
3.96375	0.210432664222397\\
3.97625	0.207049316896621\\
3.99125	0.202439867014446\\
3.99875	0.200008091740751\\
};
\addplot [color=black, line width=1.5pt, forget plot]
  table[row sep=crcr]{%
-1.49875	0.5\\
2.07125	0.499790353633827\\
2.09375	0.499256613692387\\
2.10875	0.498413858095128\\
2.11875	0.49746654480382\\
2.12875	0.49606752934034\\
2.13625	0.494629890283451\\
2.14375	0.492776406561253\\
2.15125	0.49042381414666\\
2.15875	0.487482100827338\\
2.16625	0.483856013957619\\
2.17125	0.481009772757557\\
2.17625	0.477785904858809\\
2.18125	0.474154838060056\\
2.18625	0.470087424907602\\
2.19125	0.465555346625842\\
2.19625	0.460531577587637\\
2.20125	0.454990929039297\\
2.20625	0.448910695334441\\
2.21125	0.442271429744145\\
2.21625	0.435057878989927\\
2.22125	0.42726010450338\\
2.22875	0.41446300984989\\
2.23625	0.400371186317957\\
2.24375	0.385064996606767\\
2.25125	0.36869418717254\\
2.26125	0.345623265745318\\
2.28375	0.292798188649859\\
2.29125	0.276528570866453\\
2.29875	0.261669010814284\\
2.30375	0.252709229728754\\
2.30875	0.244581982765223\\
2.31375	0.23731850699215\\
2.31875	0.230921060890893\\
2.32375	0.225365880856655\\
2.32875	0.220607755410876\\
2.33375	0.216585490237855\\
2.33875	0.213227551202286\\
2.34375	0.210457302789834\\
2.34875	0.208197450045044\\
2.35375	0.206373489403207\\
2.36125	0.204305108295159\\
2.36875	0.202858491664244\\
2.37625	0.201866753378998\\
2.38625	0.201031512577016\\
2.39875	0.200472501021916\\
2.41875	0.200123950171369\\
2.46125	0.200005033381073\\
3.10125	0.2\\
3.99875	0.2\\
};
\addplot [color=red, dotted, line width=1.5pt, forget plot]
  table[row sep=crcr]{%
3.5025	0\\
3.5025	0.612554213014639\\
};
\end{axis}
\end{tikzpicture}%
\input{Plotsnew/microdensityT20}
\begin{tikzpicture}

\begin{axis}[%
width=0.951\fwidth,
height=0.5\fwidth,
at={(0\fwidth,0\fwidth)},
scale only axis,
xmin=7.5,
xmax=11.5,
ymin=0.1,
ymax=0.6,
axis background/.style={fill=white},
title style={font=\bfseries},
title={$T=20$}
]
\addplot [color=blue, dashed, line width=1.5pt, forget plot]
  table[row sep=crcr]{%
-1.49875	0\\
7.86875	0.00020264035170392\\
7.87375	0.000554610532303101\\
7.87875	0.00139923819067711\\
7.88125	0.002155913508334\\
7.88375	0.00325533571582426\\
7.88625	0.00481745006476508\\
7.88875	0.0069877378726666\\
7.89125	0.0099358796764708\\
7.89375	0.0138513874863282\\
7.89625	0.0189357378183441\\
7.89875	0.025390967985194\\
7.90125	0.0334052775931646\\
7.90375	0.0431368133322714\\
7.90625	0.0546973714815238\\
7.90875	0.0681380788656565\\
7.91125	0.0834390901322291\\
7.91625	0.119166301581613\\
7.92125	0.160284716602643\\
7.93875	0.311259276271302\\
7.94375	0.347726091571488\\
7.94875	0.379235946784128\\
7.95375	0.405716574848325\\
7.95625	0.417159449569109\\
7.95875	0.427483825766179\\
7.96125	0.43676141747541\\
7.96375	0.445068779610494\\
7.96625	0.452484296374454\\
7.96875	0.459085850080964\\
7.97125	0.464949087750052\\
7.97375	0.470146193942876\\
7.97625	0.47474507891587\\
7.97875	0.478808897692288\\
7.98125	0.482395825344202\\
7.98625	0.488346753870468\\
7.99125	0.492965599145469\\
7.99625	0.496549480428646\\
8.00125	0.499334689022005\\
8.00625	0.501506918647882\\
8.01125	0.503210769247346\\
8.01875	0.505125617423911\\
8.02625	0.506506066578936\\
8.03625	0.507821412030399\\
8.04875	0.508974669500306\\
8.06875	0.5102625739815\\
8.09625	0.511500353847252\\
8.12625	0.512364533743089\\
8.15625	0.512766352974024\\
8.19125	0.51275765477606\\
8.23625	0.512237968420816\\
8.29125	0.511115112459901\\
8.34875	0.509476622415304\\
8.46125	0.506049765793273\\
8.48625	0.505919980833124\\
8.51125	0.506266433435075\\
8.53875	0.507155675309866\\
8.63625	0.510860540725938\\
8.66875	0.511258187135329\\
8.70625	0.511204258008794\\
8.75375	0.51063531810413\\
8.80875	0.509509342795605\\
8.86375	0.507910856339485\\
8.97375	0.504436058330933\\
8.99875	0.504340436603716\\
9.02125	0.504696964457812\\
9.04875	0.505619758597383\\
9.10875	0.507865061492865\\
9.13125	0.508153516715499\\
9.15125	0.507969061264472\\
9.17125	0.507302265337096\\
9.18875	0.506271169318403\\
9.20625	0.504764144319061\\
9.22125	0.503040140209382\\
9.23625	0.500860442191108\\
9.25125	0.498163318322128\\
9.26375	0.495471609985238\\
9.27625	0.492331825216137\\
9.28875	0.488701050504494\\
9.30125	0.484536012472713\\
9.31375	0.479794007980264\\
9.32375	0.475557263283669\\
9.33375	0.47090498561376\\
9.34375	0.465818763140684\\
9.35375	0.460282118547456\\
9.36375	0.454281030172211\\
9.37375	0.447804541392566\\
9.38375	0.440845477168518\\
9.39375	0.433401283898133\\
9.40375	0.425475001387872\\
9.41625	0.414904770594147\\
9.42875	0.403633752335734\\
9.44125	0.391721676189235\\
9.45625	0.376704369195622\\
9.47375	0.358451708548232\\
9.51875	0.31102988895381\\
9.53125	0.298639602782451\\
9.54375	0.286950244\\
9.55375	0.278209078188985\\
9.56375	0.270078261353889\\
9.57375	0.26260104323566\\
9.58375	0.255801656831867\\
9.59375	0.249685879179724\\
9.60375	0.244242664082098\\
9.61375	0.239446570169754\\
9.62375	0.235260667241095\\
9.63375	0.231639613886687\\
9.64375	0.228532646290386\\
9.65375	0.225886284967091\\
9.66625	0.223144074096661\\
9.67875	0.220935733488387\\
9.69125	0.219166964281055\\
9.70625	0.21750658975043\\
9.72375	0.216051313178943\\
9.74625	0.214704610735986\\
9.77625	0.213450769496152\\
9.82375	0.212010821069406\\
9.91875	0.209660895456802\\
9.98125	0.20853021563325\\
10.04125	0.207916073090226\\
10.13875	0.207482829110734\\
10.25375	0.206839941660453\\
10.33125	0.205944180794203\\
10.43125	0.204249922518366\\
10.50625	0.203171240227135\\
10.57125	0.202735112609689\\
10.68875	0.202534892447305\\
10.83125	0.202147625212993\\
10.93375	0.201358170977032\\
11.03875	0.200614444625856\\
11.15375	0.200457839265976\\
11.45625	0.200140186766612\\
11.49875	0.200000085354377\\
};
\addplot [color=black, line width=1.5pt, forget plot]
  table[row sep=crcr]{%
-1.49875	0.5\\
9.07125	0.499771123197435\\
9.11625	0.499261689696024\\
9.14625	0.498502651031785\\
9.17125	0.497416430402543\\
9.19125	0.496112896985816\\
9.20875	0.494551610858288\\
9.22625	0.492498819144185\\
9.24125	0.490269114411674\\
9.25625	0.487531392853153\\
9.26875	0.484809250214745\\
9.28125	0.481640636777579\\
9.29375	0.477982652633079\\
9.30625	0.473793228488566\\
9.31875	0.469031928269567\\
9.33125	0.463660817813786\\
9.34125	0.458901550403688\\
9.35125	0.453714211399022\\
9.36125	0.448085380565733\\
9.37125	0.442004271360409\\
9.38125	0.435463193935259\\
9.39125	0.428458071432232\\
9.40125	0.42098901654453\\
9.41375	0.411008342758615\\
9.42625	0.400332954029976\\
9.43875	0.389000830745033\\
9.45125	0.377068445305913\\
9.46625	0.36206802780352\\
9.48375	0.343858955171983\\
9.53375	0.291228282566909\\
9.54625	0.278923295616742\\
9.55875	0.267346659026146\\
9.56875	0.258721279363169\\
9.57875	0.250734831501049\\
9.58875	0.243435365352944\\
9.59875	0.236850194803612\\
9.60875	0.230985697192237\\
9.61875	0.22582867733929\\
9.62875	0.221349025155213\\
9.63875	0.217503242650844\\
9.64875	0.214238360437873\\
9.65875	0.2114958003274\\
9.66875	0.209214841688816\\
9.68125	0.206921604875903\\
9.69375	0.205145610570316\\
9.70875	0.203558250988632\\
9.72625	0.202274094979218\\
9.74625	0.201333520780878\\
9.77125	0.200662460234888\\
9.80625	0.200234300666873\\
9.86875	0.200030746407553\\
10.12875	0.200000000598763\\
11.49875	0.199999999999999\\
};
\addplot [color=red, dotted, line width=1.5pt, forget plot]
  table[row sep=crcr]{%
11.0025	0\\
11.0025	0.612554213015205\\
};
\end{axis}

\end{tikzpicture}%
\end{center}
\caption{Solution of the microscopic (left column) and macroscopic (right column) density (blue dashed line) compared to the equilibrium density (black solid line). The vertical red line gives the position of the car $x_{N}(t)$ (microscopic) or the boundary $\alpha(t)$ (macroscopic).}\label{fig:timesnapshot}
\end{figure}
}\\
This is shown in Figure \ref{fig:lyapunov}, too, which displays the logarithm of the Lyapunov functions \eqref{eq:microLyapunov} and \eqref{eq:macroLyapunov} respectively together with the theoretical upper bounds.
Here, \rv{also the following kernel functions, in line with \eqref{eq:ass:kernel}, are considered:
\begin{align*}
\wt^{\text{lin.}}(x)=2\frac{\ndt-x}{\ndt^2},\qquad \wt^{\text{lin2.}}(x)=\frac{3\ndt-2x}{2\ndt^2},\qquad \wt^{\text{conc.}}(x)=3\frac{\ndt^2-x^2}{\ndt^3}.
\end{align*}
Note that for the second linear kernel function $\wt^{\text{lin2.}}(\ndt)>0$ holds.
This plays a particular role when we consider the theoretical upper bounds from the Theorems \ref{thm:microconcave} and \ref{thm:macroLyapunovDensity}.
Here, the rates are only slightly decreasing for $\wt^{\text{lin.}}$ and $\wt^{\text{conc.}}$.
As the initial conditions are greater than the equilibrium density, the distance between the last car of the bulk and the leading vehicle will be $\ndt$ in the limit, i.e. $\lim_{t\to\infty} x_N(t)-x_J(t)=\ndt$ and $\lim_{t\to\infty} \beta(t)-\alpha(t)=\ndt$. 
In addition, we have $\wt^{\text{lin.}}(\ndt)=\wt^{\text{conc.}}(\ndt)=0$ which results in a slight decrease.
We note that this only occurs, if $\wt(\ndt)=0$ and the initial data are larger than the equilibrium density.
In case of increasing initial data smaller than the equilibrium density the bounds presented in Theorems \ref{thm:microconcave} and \ref{thm:macroLyapunovDensity} are sharper.
\begin{figure}
\setlength{\fwidth}{0.4\textwidth}
\begin{center}
\begin{tikzpicture}

\begin{axis}[%
width=0.951\fwidth,
height=0.75\fwidth,
at={(0\fwidth,0\fwidth)},
scale only axis,
xmin=0,
xmax=20,
ymin=-20,
ymax=-5,
axis background/.style={fill=white},
title style={font=\bfseries},
axis x line*=bottom,
axis y line*=left,
legend pos=south west
]
\addplot [color=black, line width=1.5pt]
  table[row sep=crcr]{%
0	-5.71790985391542\\
0.559999999999999	-6.06125577000972\\
1.19	-6.43411581020959\\
1.91	-6.84702574026103\\
2.75	-7.31539046016623\\
3.73	-7.84844879834403\\
5.16	-8.60922601191113\\
6.66	-9.39403616460369\\
9.15	-10.6794043303495\\
10.19	-11.2157333286989\\
10.99	-11.6232153125419\\
11.85	-12.0618350906864\\
12.99	-12.6425268830344\\
13.84	-13.0743306097687\\
15.01	-13.6684150219499\\
15.86	-14.0994814253737\\
17.02	-14.6872547542906\\
17.87	-15.1178556931096\\
19.12	-15.7509955278871\\
20	-16.1941805115611\\
};
\addlegendentry{$\wt^{\text{const.}}$};
\addplot [color=gray, line width=1.5pt, forget plot]
  table[row sep=crcr]{%
0	-5.71790985391542\\
20	-15.7179098539154\\
};
\addplot [color=black, dotted, line width=1.5pt]
  table[row sep=crcr]{%
0	-5.71790985391542\\
0.32	-5.97898772373653\\
0.68	-6.25777049941597\\
1.09	-6.56068669737573\\
1.57	-6.90068231780461\\
2.14	-7.2897437359121\\
2.83	-7.74620112113588\\
3.77	-8.35268426976061\\
5.25	-9.29299256886661\\
7.78	-10.9096496399415\\
8.46	-11.3568438219423\\
9.51	-12.0482558595658\\
10.07	-12.4191003787117\\
10.93	-13.0049857100899\\
11.52	-13.4003046221113\\
12.05	-13.7606348177159\\
12.83	-14.3068867107589\\
14.42	-15.4306571484414\\
15.35	-16.1049961339278\\
15.83	-16.4502964569754\\
16.42	-16.8898477522391\\
17.27	-17.5232382813189\\
17.85	-17.946994319133\\
18.43	-18.3877768436433\\
19.25	-19.0108982227992\\
20	-19.5610338692331\\
};
\addlegendentry{$\wt^{\text{lin.}}$};
\addplot [color=gray, dotted, line width=1.5pt, forget plot]
  table[row sep=crcr]{%
0	-5.71790985391542\\
0.34	-5.82813170169827\\
0.699999999999999	-5.93261339433719\\
1.08	-6.03091355558875\\
1.49	-6.12489610160304\\
1.92	-6.21173293513677\\
2.39	-6.29476611704805\\
2.89	-6.37140030563987\\
3.44	-6.44386022793941\\
4.04	-6.5110291271493\\
4.7	-6.57307062242523\\
5.43	-6.62995015785286\\
6.25	-6.68214826724417\\
7.19	-6.73021848604811\\
8.3	-6.77507498082326\\
9.63	-6.81686083815638\\
11.32	-6.85806206371324\\
13.67	-6.90305467107095\\
17.45	-6.96263076280185\\
20	-6.99966244564832\\
};
\addplot [color=black, dashdotted, line width=1.5pt]
  table[row sep=crcr]{%
0	-5.71790985391542\\
0.399999999999999	-6.00371466592889\\
0.850000000000001	-6.31090673435068\\
1.36	-6.64474465420149\\
1.93	-7.00399822469669\\
2.58	-7.4001560695805\\
3.33	-7.84400544071354\\
4.22	-8.35753623307169\\
5.41	-9.03078756136414\\
6.77	-9.7901609483264\\
7.72	-10.3163306531726\\
8.54	-10.7731765232539\\
9.63	-11.3796544065059\\
10.36	-11.7850589938914\\
13.11	-13.3383006987349\\
16.17	-15.0872270023102\\
18.83	-16.6403180519334\\
19.6	-17.0906295190179\\
20	-17.3221698040456\\
};
\addlegendentry{$\wt^{\text{lin2.}}$};
\addplot [color=gray, dashdotted, line width=1.5pt, forget plot]
  table[row sep=crcr]{%
0	-5.71790985391542\\
0.52	-5.93073181614846\\
1.09	-6.15138829209693\\
1.71	-6.37881148520346\\
2.39	-6.61572328172231\\
3.14	-6.86457814759573\\
3.97	-7.12770872980651\\
4.91	-7.41341455769188\\
5.98	-7.72639773941748\\
7.23	-8.07978114248955\\
8.73	-8.49152524560438\\
10.61	-8.99509967990636\\
13.11	-9.65219508883539\\
16.89	-10.6327216465531\\
20	-11.4346425997283\\
};
\addplot [color=black, dashed, line width=1.5pt]
  table[row sep=crcr]{%
0	-5.71790985391542\\
0.359999999999999	-5.99941551651853\\
0.760000000000002	-6.29713726839024\\
1.2	-6.60980783428876\\
1.68	-6.93682573646177\\
2.22	-7.29092091397948\\
2.84	-7.68365582282006\\
3.55	-8.11997950785885\\
4.42	-8.64152046231973\\
5.77	-9.43362579707772\\
11.28	-12.6621974182778\\
12.19	-13.2117249336122\\
13.95	-14.2836809489401\\
14.75	-14.7755359801699\\
15.35	-15.1457833800877\\
17.83	-16.7069519731857\\
18.67	-17.2406929059792\\
19.35	-17.6657403952415\\
20	-18.0761981332185\\
};
\addlegendentry{$\wt^{\text{conc.}}$};
\addplot [color=gray, dashed, line width=1.5pt, forget plot]
  table[row sep=crcr]{%
0	-5.71790985391542\\
0.350000000000001	-5.86059879529619\\
0.710000000000001	-5.99504671761363\\
1.09	-6.12448829372706\\
1.48	-6.24515635014748\\
1.89	-6.35993493541488\\
2.32	-6.46831486209627\\
2.77	-6.56995167800483\\
3.25	-6.66655013019525\\
3.76	-6.75740137257383\\
4.31	-6.84349382988642\\
4.9	-6.9240009066866\\
5.54	-6.99948928119195\\
6.23	-7.06928057643608\\
7	-7.1354006803944\\
7.85	-7.19677787251917\\
8.81	-7.25441579653547\\
9.93	-7.31000930973282\\
11.26	-7.36415522735522\\
12.9	-7.41907314305248\\
15.05	-7.47904017287179\\
18.2	-7.55449762779419\\
20	-7.59450258098102\\
};
\end{axis}

\end{tikzpicture}%
\begin{tikzpicture}

\begin{axis}[%
width=0.951\fwidth,
height=0.75\fwidth,
at={(0\fwidth,0\fwidth)},
scale only axis,
xmin=0,
xmax=20,
ymin=-20,
ymax=-5,
axis background/.style={fill=white},
title style={font=\bfseries},
axis x line*=bottom,
axis y line*=left
]
\addplot [color=black, line width=1.5pt, forget plot]
  table[row sep=crcr]{%
0	-5.70176560894742\\
0.0199999999999996	-5.71973016023912\\
0.640000000000001	-6.13326464692381\\
1.24	-6.51058736058809\\
1.82	-6.86005712312295\\
2.52	-7.26825559215598\\
3.35	-7.73876181452612\\
4.33	-8.28081248262547\\
5.47	-8.89805217286008\\
6.77	-9.58873125126827\\
8.21	-10.3406155841329\\
9.77	-11.142128682219\\
11.49	-12.0127329571687\\
13.38	-12.9561246888612\\
15.44	-13.9713032241804\\
17.7	-15.071936315944\\
20	-16.1804788415236\\
};
\addplot [color=gray, line width=1.5pt, forget plot]
  table[row sep=crcr]{%
0	-5.70176560894742\\
20	-15.7017656089474\\
};
\addplot [color=black, dotted, line width=1.5pt, forget plot]
  table[row sep=crcr]{%
0	-5.70346720195187\\
0.0300000000000011	-5.73534955024023\\
0.190000000000001	-5.87964662739675\\
0.370000000000001	-6.0326092134091\\
0.670000000000002	-6.27559082335443\\
1.03	-6.55195808684185\\
1.47	-6.87504824754798\\
1.99	-7.24224768922502\\
2.64	-7.68638390979481\\
3.46	-8.23199089504353\\
4.59	-8.96961393374382\\
8.85	-11.738791738087\\
10.43	-12.7853896659361\\
11.95	-13.8058412031616\\
13.55	-14.8940377326736\\
15.2	-16.0303621636132\\
16.95	-17.249575767265\\
18.84	-18.5804643619771\\
20	-19.4036666707076\\
};
\addplot [color=gray, dotted, line width=1.5pt, forget plot]
  table[row sep=crcr]{%
0	-5.70346720195187\\
0.300000000000001	-5.79811720195188\\
0.620000000000001	-5.88821720195187\\
0.989999999999998	-5.98051720195187\\
1.36	-6.06176720195187\\
1.79	-6.14426720195187\\
2.21	-6.21416720195187\\
2.7	-6.28441720195188\\
3.28	-6.35416720195187\\
3.84	-6.41016720195188\\
4.53	-6.46661720195188\\
5.15	-6.50751720195187\\
5.94	-6.54926720195187\\
6.98	-6.59041720195187\\
7.9	-6.61691720195187\\
9.22	-6.64426720195187\\
11.36	-6.67066720195186\\
13.26	-6.68396720195186\\
20	-6.71766720195183\\
};
\addplot [color=black, dashdotted, line width=1.5pt, forget plot]
  table[row sep=crcr]{%
0	-5.70282411150368\\
0.0199999999999996	-5.7231293184297\\
0.350000000000001	-5.9810673517169\\
0.949999999999999	-6.41216539268109\\
1.4	-6.71649922214549\\
1.93	-7.06075435787353\\
2.54	-7.4432031451985\\
3.26	-7.88099087113645\\
4.12	-8.39029933186514\\
5.2	-9.01579326730993\\
6.6	-9.81224711491306\\
8.61	-10.9414860012389\\
14.73	-14.374358579581\\
17.44	-15.9116100821904\\
20	-17.3755763373519\\
};
\addplot [color=gray, dashdotted, line width=1.5pt, forget plot]
  table[row sep=crcr]{%
0	-5.70282411150368\\
0.460000000000001	-5.88892411150368\\
1.01	-6.09884911150368\\
1.62	-6.31854911150368\\
2.29	-6.54709911150368\\
3.02	-6.78382411150368\\
3.79	-7.02239911150368\\
4.78	-7.31599911150368\\
5.89	-7.63199911150368\\
7.1	-7.96492411150367\\
8.86	-8.43427411150368\\
11.06	-9.00587411150369\\
13.89	-9.72869911150365\\
20	-11.2714741115037\\
};
\addplot [color=black, dashed, line width=1.5pt, forget plot]
  table[row sep=crcr]{%
0	-5.70316390614661\\
0.0300000000000011	-5.73360170328806\\
0.23	-5.90607105467494\\
0.510000000000002	-6.13348707632656\\
0.91	-6.4411500279749\\
1.43	-6.81844498731705\\
1.91	-7.15051475828348\\
2.47	-7.52362136728949\\
3.11	-7.93615495907046\\
3.89	-8.424758361188\\
4.86	-9.01804456661424\\
6.22	-9.834764383069\\
8.96	-11.4638332584228\\
11.46	-12.9570589399425\\
13.63	-14.2663733347865\\
15.82	-15.6010582511768\\
18.15	-17.0346242329669\\
20	-18.1814877634641\\
};
\addplot [color=gray, dashed, line width=1.5pt, forget plot]
  table[row sep=crcr]{%
0	-5.70316390614661\\
0.309999999999999	-5.82669828114661\\
0.629999999999999	-5.94317871864661\\
1.01	-6.06817896864661\\
1.38	-6.17753159364661\\
1.75	-6.27591453114661\\
2.17	-6.37588428114661\\
2.65	-6.47621459364661\\
3.11	-6.56026915614661\\
3.65	-6.64592365614661\\
4.11	-6.70896434364661\\
4.66	-6.77424546864661\\
5.33	-6.84121728114661\\
6.15	-6.90706821864661\\
6.83	-6.95107034364661\\
7.69	-6.99592990614661\\
8.86	-7.04217284364662\\
10.5	-7.08520690614662\\
11.74	-7.10719484364661\\
13.77	-7.13183534364661\\
20	-7.17844353114661\\
};
\end{axis}

\end{tikzpicture}%
\end{center}
\caption{Microscopic and macroscopic Lyapunov functions (black) over time together with the exponential upper bounds (gray) for different kernel functions.}\label{fig:lyapunov}
\end{figure}
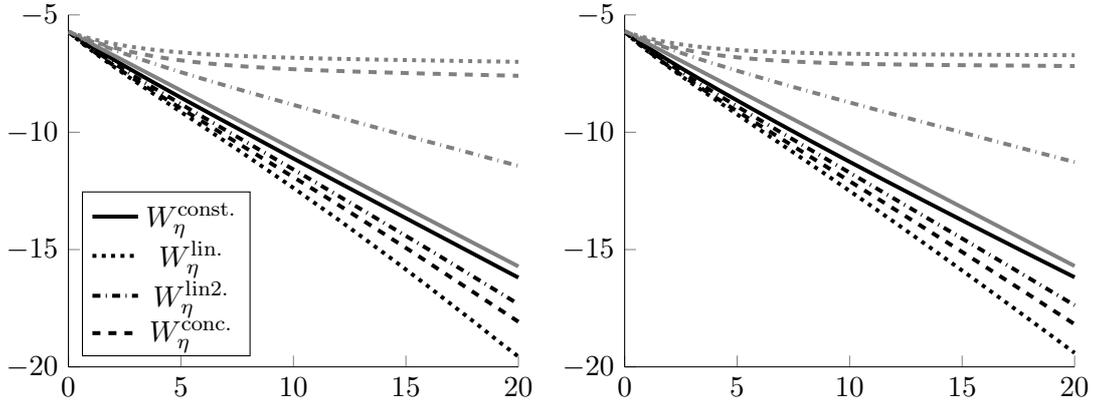
\begin{remark}
Note that the exponential bounds strongly depend on the values of $\rho_{\min}$ and hence on the choice of the initial conditions. Using the same parameters as before, but consider the constant kernel function $\wt^{\text{const.}}$ and  initial data 
\begin{align*}
\rho_0(x)=\begin{cases}
0.5,&\text{if }x<0,\\
0.025,&\text{if }x\in[0,0.75],\\
0.05,&\text{if }x>0.75.
\end{cases}
\end{align*}
we observe the Lyapunov function and its upper bound in Figure \ref{fig:lyapunovlowdensity}, left.
In particular, it is obvious that the bound is not  sharp, e.g., by comparing Figure \ref{fig:lyapunovlowdensity} to Figure \ref{fig:lyapunov}.
Figure \ref{fig:lyapunovlowdensity}, right, shows the Lyapunov function of \rv{Theorem \ref{thm:microconstant}} for all cars $i=0,\dots,N-1$.
This function has an increasing part at the beginning which demonstrates that generalizing the results of \rv{Theorem \ref{thm:microconstant}} to all cars cannot be expected.
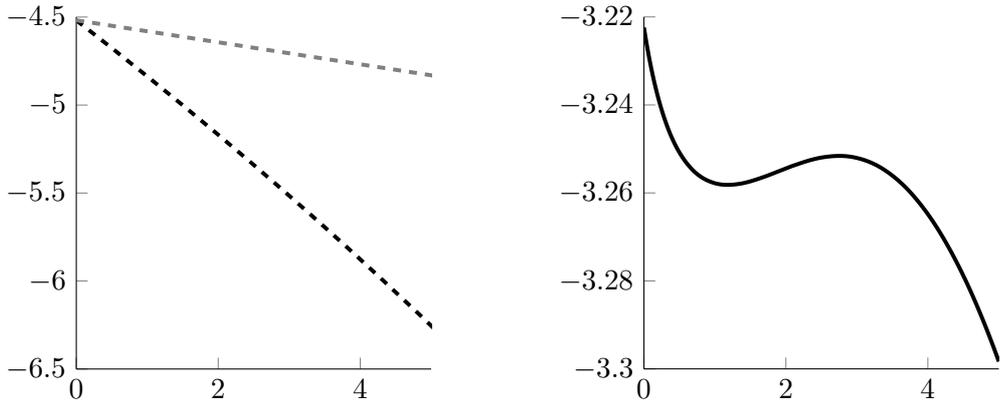
\begin{figure}
\setlength{\fwidth}{0.6\textwidth}
\begin{center}
\begin{tikzpicture}

\begin{axis}[%
width=0.5\fwidth,
height=0.5\fwidth,
at={(0\fwidth,0\fwidth)},
scale only axis,
xmin=0,
xmax=5,
ymin=-6.5,
ymax=-4.5,
axis background/.style={fill=white},
axis x line*=bottom,
axis y line*=left
]
\addplot [color=black, dashed, line width=1.5pt, forget plot]
  table[row sep=crcr]{%
0	-4.51842500173924\\
0.51	-4.67938214828871\\
1.02	-4.84359538676269\\
1.5	-5.00127508693059\\
1.98	-5.16214743792667\\
2.46	-5.32631583062387\\
2.92	-5.48687523360515\\
3.4	-5.65781649929966\\
3.89	-5.83590364408169\\
4.32	-5.9958164970479\\
4.83	-6.18866384668959\\
5.28	-6.36270005035785\\
5.77	-6.55522829804278\\
6.26	-6.75183596969991\\
6.64	-6.90643673716858\\
7.14	-7.11349194943088\\
7.52	-7.27283923556715\\
8.05	-7.49876242418608\\
8.45	-7.67143385448647\\
9.06	-7.9386188170809\\
9.51	-8.13870782043685\\
10	-8.35865252584326\\
};
\addplot [color=gray, dashed, line width=1.5pt, forget plot]
  table[row sep=crcr]{%
0	-4.51842500173924\\
10	-5.14342500173924\\
};

\end{axis}

\begin{axis}[%
width=0.5\fwidth,
height=0.5\fwidth,
at={(0.8\fwidth,0\fwidth)},
scale only axis,
xmin=0,
xmax=5,
ymin=-3.3,
ymax=-3.22,
axis background/.style={fill=white},
axis x line*=bottom,
axis y line*=left
]
\addplot [color=black, line width=1.5pt, forget plot]
  table[row sep=crcr]{%
0	-3.22241205282565\\
0.0199999999999996	-3.2244250416479\\
0.04	-3.22632891154836\\
0.0599999999999996	-3.22813029935643\\
0.0800000000000001	-3.22983690677248\\
0.0999999999999996	-3.23145248128295\\
0.12	-3.2329822315814\\
0.14	-3.23443110850047\\
0.16	-3.23580380922566\\
0.18	-3.23710478140481\\
0.2	-3.23833822714845\\
0.22	-3.23950810691374\\
0.24	-3.24061814326366\\
0.26	-3.24167182449136\\
0.28	-3.24267240809839\\
0.3	-3.24362292411462\\
0.33	-3.24496648472301\\
0.35	-3.245805529495\\
0.37	-3.24659989581936\\
0.39	-3.24735194594805\\
0.42	-3.24840558791458\\
0.45	-3.24937625946206\\
0.48	-3.25027064063653\\
0.51	-3.25109491978334\\
0.54	-3.25185480750805\\
0.57	-3.25255555009032\\
0.6	-3.25320194232193\\
0.63	-3.25379833974089\\
0.67	-3.25452723159323\\
0.7	-3.2550175611351\\
0.73	-3.25545984615231\\
0.77	-3.25598099105554\\
0.81	-3.25643098266349\\
0.85	-3.25681707219246\\
0.89	-3.25714587577706\\
0.94	-3.25748540848776\\
0.99	-3.25775746944561\\
1.03	-3.25792818844135\\
1.08	-3.25807395067934\\
1.13	-3.25815361845973\\
1.18	-3.25817614378315\\
1.24	-3.25813922052248\\
1.3	-3.25804450735144\\
1.37	-3.25787616288057\\
1.43	-3.25767027884209\\
1.5	-3.25737130850308\\
1.59	-3.25692013898285\\
1.72	-3.25619831887099\\
1.83	-3.25551518317246\\
2.08	-3.25393752902628\\
2.25	-3.25300136443703\\
2.34	-3.25256208565433\\
2.42	-3.25222817573802\\
2.49	-3.25199158955255\\
2.57	-3.25179484212588\\
2.65	-3.25165699419651\\
2.72	-3.25159439751238\\
2.79	-3.25159488119097\\
2.85	-3.25165166770236\\
2.91	-3.25176566523473\\
2.97	-3.25194474619288\\
3.04	-3.25221648687257\\
3.1	-3.25250663268999\\
3.16	-3.25285425734756\\
3.22	-3.25326362779021\\
3.28	-3.25373887883252\\
3.34	-3.25428401842447\\
3.4	-3.25490143374757\\
3.46	-3.25557863565487\\
3.52	-3.25631908785394\\
3.58	-3.2571263826195\\
3.64	-3.25800402564979\\
3.69	-3.25879160533927\\
3.74	-3.25963235639472\\
3.79	-3.26052917086312\\
3.85	-3.26166981812096\\
3.91	-3.26287848377665\\
3.97	-3.26415816456705\\
4.02	-3.26528093807398\\
4.07	-3.26645674945634\\
4.12	-3.26768725254427\\
4.17	-3.26897407939418\\
4.22	-3.27031901606795\\
4.28	-3.27200006998115\\
4.33	-3.2734565774093\\
4.38	-3.2749652383708\\
4.43	-3.27652752105959\\
4.48	-3.27814488080262\\
4.53	-3.27981876095606\\
4.58	-3.28155059376965\\
4.63	-3.283343056182\\
4.68	-3.28519213856917\\
4.73	-3.28709396380284\\
4.78	-3.2890498394994\\
4.83	-3.29106106826051\\
4.88	-3.29312894845886\\
4.93	-3.2952547750129\\
4.98	-3.29743984015006\\
5	-3.29833074208074\\
};
\end{axis}

\end{tikzpicture}%
\end{center}
\caption{Example in which the exponential bound is not so sharp due to a low value of $\rho_{\min}$, left, and the Lyapunov function for all cars, right.}\label{fig:lyapunovlowdensity}
\end{figure}
\end{remark}
Finally, we consider the following example
\begin{align*}
\rho_0(x)=\begin{cases}
0.5,&\text{if }x<0,\\
0.3,&\text{if }x\in[0,0.5],\\
0.4,&\text{if }x>0.5,
\end{cases}
\qquad\text{and}\qquad
\lm_0(x)=\rv{\begin{cases}
1,&\text{if }x<0,\\
\frac{5}{8},&\text{if }x\in[0,0.5],\\
\frac{3}{4},&\text{if }x>0.5.
\end{cases}}
\end{align*}
with $\ndt=2$ and an initial placement of the cars in $[-5,1]$.
Note that in this situation our stabilization results hold only for the constant kernel. 
Nevertheless, the convergence towards the steady state in the micro- and macroscopic scale is also obtained for concave kernels and even convex ones, i.e.
\[\wt^{\text{conv.}}=3\frac{(\ndt-x)^2}{\ndt^3}.\]
This is shown in Figure \ref{fig:lyapunovnonlinear}.
Even though not covered by our theoretical results, the bounds recovered in Theorems \ref{thm:microconcave} and \ref{thm:macroLyapunovDensity} are valid for the considered example and the linear kernel, too.
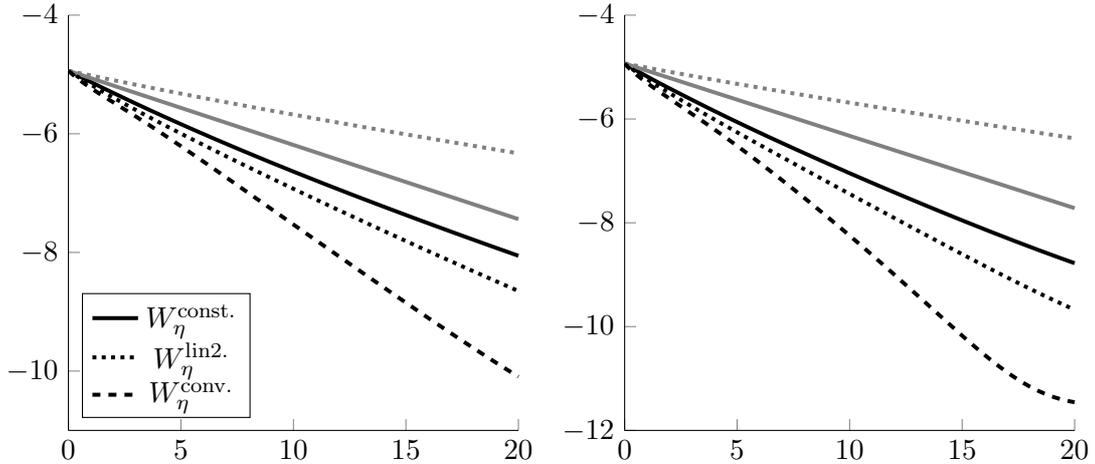
\begin{figure}
\setlength{\fwidth}{0.4\textwidth}
\begin{center}
\begin{tikzpicture}

\begin{axis}[%
width=0.951\fwidth,
height=0.885\fwidth,
at={(0\fwidth,0\fwidth)},
scale only axis,
xmin=0,
xmax=20,
ymin=-11,
ymax=-4,
axis background/.style={fill=white},
title style={font=\bfseries},
axis x line*=bottom,
axis y line*=left,
legend pos= south west
]
\addplot [color=black, line width=1.5pt]
  table[row sep=crcr]{%
0	-4.93812244027054\\
0.82	-5.09453220210782\\
1.65	-5.24897021363543\\
2.64	-5.42848894800811\\
3.53	-5.58592813430852\\
4.51	-5.75491919058814\\
5.49	-5.92018662658454\\
6.55	-6.0945963923888\\
7.6	-6.26366814681888\\
8.89	-6.46653815659178\\
10.01	-6.63867579894375\\
11.33	-6.83758605774509\\
12.65	-7.03217765352732\\
13.96	-7.22162423833255\\
15.48	-7.43741659413248\\
17.24	-7.68213922442983\\
18.85	-7.90166436235328\\
20	-8.0561123641801\\
};
\addlegendentry{$\wt^{\text{const.}}$};
\addplot [color=gray, line width=1.5pt, forget plot]
  table[row sep=crcr]{%
0	-4.93812244027054\\
20	-7.43812244027054\\
};
\addplot [color=black, dotted, line width=1.5pt]
  table[row sep=crcr]{%
0	-4.93812244027054\\
0.440000000000001	-5.043471830159\\
0.93	-5.15662031621604\\
1.45	-5.27256565619699\\
2.08	-5.40838266420244\\
2.74	-5.5465377256829\\
3.46	-5.69312990938323\\
4.34	-5.86813265167935\\
5.27	-6.04896201182466\\
6.39	-6.262610991406\\
7.57	-6.48349243644982\\
8.81	-6.71184190667011\\
10.24	-6.97112590634735\\
11.88	-7.26371550558424\\
13.68	-7.57955722982259\\
15.53	-7.89898632581023\\
17.35	-8.20842856373357\\
19.12	-8.50483376145859\\
20	-8.65039078453273\\
};
\addlegendentry{$\wt^{\text{lin2.}}$};
\addplot [color=gray, dotted, line width=1.5pt, forget plot]
  table[row sep=crcr]{%
0	-4.93812244027054\\
1.58	-5.06399378098207\\
3.19	-5.18825786027553\\
4.88	-5.31468737009482\\
6.68	-5.44533026141655\\
8.63	-5.58282458926583\\
10.76	-5.72897170475055\\
13.12	-5.88684688318336\\
15.76	-6.05939470125379\\
18.76	-6.25139748518513\\
20	-6.32973405668715\\
};
\addplot [color=black, dashed, line width=1.5pt]
  table[row sep=crcr]{%
0	-4.93812244027054\\
0.170000000000002	-4.99541780605366\\
0.350000000000001	-5.0519022804783\\
0.550000000000001	-5.1105023826906\\
0.780000000000001	-5.17351845814759\\
1.06	-5.24587130440668\\
1.39	-5.32696066807091\\
1.99	-5.4694139495694\\
2.91	-5.6875538863228\\
3.61	-5.85810597003239\\
4.35	-6.04319471021263\\
5.17	-6.25290787430878\\
6.16	-6.51079267877669\\
7.48	-6.85913210473375\\
10.42	-7.64059324237785\\
13.01	-8.32640249927348\\
14.22	-8.64404236701441\\
15.77	-9.04734732212723\\
17.34	-9.44776505469826\\
18.07	-9.62965160379456\\
19	-9.85739141820407\\
19.67	-10.0174530416938\\
20	-10.0945418569086\\
};
\addlegendentry{$\wt^{\text{conv.}}$};

\end{axis}

\end{tikzpicture}%
\begin{tikzpicture}

\begin{axis}[%
width=0.951\fwidth,
height=0.885\fwidth,
at={(0\fwidth,0\fwidth)},
scale only axis,
xmin=0,
xmax=20,
ymin=-12,
ymax=-4,
axis background/.style={fill=white},
title style={font=\bfseries},
axis x line*=bottom,
axis y line*=left
]
\addplot [color=black, line width=1.5pt, forget plot]
  table[row sep=crcr]{%
0	-4.93213223260831\\
0.329999999999998	-5.01928784082877\\
0.73	-5.11806025553611\\
1.22	-5.23441111225699\\
1.76	-5.35783348747391\\
2.87	-5.60441800193081\\
3.74	-5.79444136604529\\
4.56	-5.96877603818045\\
5.46	-6.15533488620917\\
6.46	-6.35788689284712\\
7.55	-6.57398589868478\\
8.72	-6.8013040149629\\
9.96	-7.03758465336777\\
11.25	-7.27873723131078\\
12.55	-7.51714219255062\\
13.83	-7.7473121022471\\
15.07	-7.96573064218883\\
16.25	-8.16902880902243\\
17.36	-8.35573881348131\\
18.41	-8.52779313862888\\
19.39	-8.68385122050383\\
20	-8.77853842930963\\
};
\addplot [color=gray, line width=1.5pt, forget plot]
  table[row sep=crcr]{%
0	-4.93213223260831\\
20	-7.71830159346601\\
};
\addplot [color=black, dotted, line width=1.5pt, forget plot]
  table[row sep=crcr]{%
0	-4.93213223260831\\
0.190000000000001	-4.99575139882381\\
0.460000000000001	-5.07896506724748\\
0.789999999999999	-5.17577739432837\\
1.16	-5.27952216027723\\
1.57	-5.38980852040376\\
2.54	-5.64469500195519\\
3.26	-5.82863448650628\\
4.05	-6.02564393864147\\
5.01	-6.2602917158761\\
6.24	-6.55597570817072\\
7.91	-6.95246332427947\\
10.27	-7.50782705649187\\
12.73	-8.08234802314485\\
14.44	-8.47735615747582\\
15.71	-8.76625278798003\\
16.74	-8.99604126959145\\
17.62	-9.18782161393363\\
18.4	-9.35321782786291\\
19.1	-9.49708491015527\\
19.74	-9.62410673953849\\
20	-9.67431780334902\\
};
\addplot [color=gray, dotted, line width=1.5pt, forget plot]
  table[row sep=crcr]{%
0	-4.93213223260831\\
1.59	-5.0608114529455\\
3.3	-5.19475914174809\\
5.17	-5.33670403492531\\
7.29	-5.49292374384528\\
9.61	-5.65940532231622\\
12.53	-5.86409089104609\\
16.16	-6.11358712974086\\
20	-6.37416537753633\\
};
\addplot [color=black, dashed, line width=1.5pt, forget plot]
  table[row sep=crcr]{%
0	-4.93213223260831\\
0.0899999999999999	-4.97559136494243\\
0.219999999999999	-5.03142969633007\\
0.379999999999999	-5.0948270581914\\
0.559999999999999	-5.16091654422154\\
0.760000000000002	-5.22931464023914\\
1.05	-5.32281428957064\\
1.4	-5.43118767106367\\
1.86	-5.56847225434471\\
3.68	-6.1067519694846\\
4.3	-6.29748666704301\\
4.92	-6.49293955145476\\
5.56	-6.6994548586907\\
6.22	-6.91718147096293\\
6.9	-7.14625421615968\\
7.6	-7.38685391604561\\
8.31	-7.63569443075036\\
9.03	-7.89289628288945\\
9.75	-8.15494961516931\\
10.47	-8.42180597537389\\
11.21	-8.70094253761096\\
11.99	-9.00012818853482\\
12.87	-9.3427025649279\\
15.04	-10.1901114248165\\
15.53	-10.3747597233028\\
15.94	-10.5244690502643\\
16.3	-10.6510868492461\\
16.62	-10.7589147100035\\
16.92	-10.8552191899145\\
17.2	-10.9403208507664\\
17.46	-11.0147362461104\\
17.71	-11.0817304419887\\
17.96	-11.1439426856206\\
18.2	-11.1989303030965\\
18.43	-11.2471126393773\\
18.66	-11.2907789614738\\
18.89	-11.3298899654292\\
19.12	-11.3644646593247\\
19.35	-11.3945794947976\\
19.59	-11.4213905475568\\
19.83	-11.443692680075\\
20	-11.4568986996436\\
};

\end{axis}
\end{tikzpicture}%
\end{center}
\caption{The considered Lyapunov functions (black) together with the exponential upper bounds (gray) for the microscopic and macroscopic scale. The assumptions are violated for the kernels $\wt^{\text{lin2.}}$ and $\wt^{\text{conv.}}$.}\label{fig:lyapunovnonlinear}
\end{figure}
}

\section{Conclusion}
In this work, we have presented suitable Lyapunov functions and explicit rates such that a nonlocal second-order model on a single road tends to its equilibrium state.
We have considered the microscopic and macroscopic scales and the rates for both scales coincide.
For the theoretical analysis, we had to restrict ourselves to a constant kernel function \rv{or a concave kernel with monotone initial data}.
Nevertheless, numerical examples suggest that the asymptotic stabilization effect can be obtained for all cars and also in the case of \rv{convex} kernels.
Future work may include extending the obtained results to those cases.
\section*{Acknowledgment}
The authors thank the Deutsche Forschungsgemeinschaft (DFG, German Research Foundation) for the financial support through 320021702/GRK2326,  333849990/IRTG-2379, B04, B05 and B06 of 442047500/SFB1481, HE5386/18-1,19-2,22-1,23-1,25-1, GO 1920/10-1, ERS SFDdM035 and under Germany’s Excellence Strategy EXC-2023 Internet of Production 390621612 and under the Excellence Strategy of the Federal Government and the Länder.

\begin{appendix}
\rv{
\section{Proof of the maximum principle for monotone initial data and concave kernel functions}
\begin{lemma}\label{lem:maxappendix}
Let the initial placement of cars and the equilibrium velocity $\bar v$ be chosen such that $J\in\{0,\dots,N-1\}$ being the smallest integer satisfying \eqref{eq:sufficient} exists.
Further, let Assumption \ref{ass:vprime} hold and the dynamics given by equation \eqref{eq:microGARZ}.
We assume that the kernel function is concave, $\lm_i=\lm$ for $i=J,\ldots,N-1$ and the initial datum satisfies either $y_J(0)\geq\dots\geq y_{N-1}(0)\geq \bar L$ or $y_J(0)\leq\dots\leq y_{N-1}(0)\leq \bar L$ with $v(\bar L,\lm)=\bar v$.
Then, the maximum principle \eqref{eq:maxmicro} holds.
\end{lemma}
\begin{proof}
Due to the assumption $\lm_i=\lm$ and for readability, we will drop the dependence of $v$ on $\lm$ in the following.
As the proof works completely analogously, we concentrate on the case $y_J(0)\leq\dots\leq y_{N-1}(0)\leq \bar L$.
Further, we assume $J\geq \mathcal{J}$.
As already outlined in the proof of Lemma \ref{lem:max} we start by proving the upper bound $\bar L$ on $y_i(t)$.
Recall that in the setting of Lemma \ref{lem:max} we prove the claim by induction.
The base case is already done in the proof of Lemma \ref{lem:max}.
Hence, we proceed with the induction step and suppose that $y_j(t)\leq \bar L$ holds for $j>i$ and therefore $v\left(\frac{1}{Ny_{j}(t)}\right)\leq\bar v$.
Further, $\gamma_{i+1,j-1}(t)-\gamma_{i,j}(t)\geq 0$ holds due to the definition of the weights and the monotonicity of the kernel function $\wt$.
We deduce from \eqref{eq:dygeneral} 
\begin{align}\label{eq:firstterm}
    \frac{d}{dt}(y_i(t)-\bar L)\leq -\gamma_{i,0}(t)\left(v\left(\frac{1}{Ny_{i}(t)}\right)-\bar v\right)=\wt(\zeta_i{(t)})\,\partial_\rho v(\xi_{i}(t)) \frac{1}{N\bar L}\left(y_{i}(t)-\bar L\right)
\end{align}
where we used $\gamma_{i,0}(t)=\int_0^{y_{i}(t)} \wt(y)dy=\wt(\zeta_{i}(t))y_{i}(t)$ and once more the mean value theorem for $v$.
By Grönwall's inequality we obtain
\begin{align*}
        y_{i}(t)-\bar L\leq\left(y_{i}(0)-\bar L\right)\exp\left(\frac{\int_0^t \wt(\zeta_{i}(s))\partial_\rho v(\xi_{i}(s)) ds}{N\bar L}\right)\leq 0.
\end{align*}
This leads to the desired upper bound $y_i(t)\leq \bar L$. To prove the lower bound, we first prove that the monotonicity of the initial data is kept for $t>0$, if  $y_i(t)\leq \bar L$ holds.
Again, we consider the time derivative
\begin{align*}
    &\frac{d}{dt}\left(y_{i+1}(t)-y_i(t)\right)\\
    =&\sum_{j=1}^{N-2-i} (\gamma_{i+2,j-1}(t)-\gamma_{i+1,j}(t))\left(v\left(\frac{1}{Ny_{i+j+1}(t)}\right)-\bar v\right)-\gamma_{i+1,0}(t)\left(v\left(\frac{1}{Ny_{i+1}\rv{(t)}}\right)-\bar v\right)\\
    &-\sum_{j=1}^{N-1-i} (\gamma_{i+1,j-1}(t)-\gamma_{i,j}(t))\left(v\left(\frac{1}{Ny_{i+j}\rv{(t)}}\right)+\bar v\right)+\gamma_{i,0}(t)\left(v\left(\frac{1}{Ny_{i}\rv{(t)}}\right)-\bar v\right)\\
    =&\gamma_{i,0}(t)\left(v\left(\frac{1}{Ny_{i}(t)}\right)-\bar v\right)-\gamma_{i+1,0}(t)\left(v\left(\frac{1}{Ny_{i+1}(t)}\right)-\bar v\right)\\ &-(\gamma_{i+1,0}(t)-\gamma_{i,1}(t))\left(v\left(\frac{1}{Ny_{i+1}(t)}\right)-\bar v\right)\\
    &+\sum_{j=2}^{N-1-i} (\gamma_{i+2,j-2}-2\gamma_{i+1,j-1}(t)+\gamma_{i,j}(t))\left(v\left(\frac{1}{Ny_{i+j}(t)}\right)-\bar v\right).
    \intertext{Due to $\gamma_{i+1,0}(t)\geq \gamma_{i,0}(t)$ and $v\left(\frac{1}{Ny_{i+1}(t)}\right)\leq\bar v$ we can drop the third term. Further, we add a zero and obtain}
    \geq &\gamma_{i,0}(t)\left(v\left(\frac{1}{Ny_{i}(t)}\right)-v\left(\frac{1}{Ny_{i+1}(t)}\right)\right)+(\gamma_{i,0}(t)-\gamma_{i+1,0}(t))\left(v\left(\frac{1}{Ny_{i+1}(t)}\right)-\bar v\right)\\
    &+\sum_{j=2}^{N-1-i} (\gamma_{i+2,j-2}-2\gamma_{i+1,j-1}(t)+\gamma_{i,j}(t))\left(v\left(\frac{1}{Ny_{i+j}(t)}\right)-\bar v\right).
\end{align*}
Next, we need to estimate everything from below depending on the distance $y_{i+1}(t)-y_i(t)$.
The first term can be handled by using the mean value theorem in a similar manner as in \eqref{eq:firstterm}.
Further, we obtain
\begin{align*}
    \gamma_{i,0}(t)-\gamma_{i+1,0}(t)&=\int_0^{y_i(t)}\wt(y)dy-\int_0^{y_{i+1}(t)}\wt(y)dy=\int_{y_{i+1}(t)}^{y_i(t)}\wt(y)dy\\
    &=-\wt(\zeta_{i+1/2}(t))(y_{i+1}(t)-y_i(t))
\end{align*}
which allows to express the second term as desired.
Finally, we consider
\begin{align*}
    &\gamma_{i+2,j-1}(t)-2\gamma_{i+1,j-1}(t)+\gamma_{i,j}(t)\\
    &=\int_{x_{i+j}(t)}^{x_{i+j+1}(t)} \wt(y-x_{i+2}(t))-2\wt(y-x_{i+1}(t))+\wt(y-x_i(t))dy
    \intertext{and use once more the mean value theorem by denoting with $\zeta_i(y,t)$ the corresponding value in $(y-x_{i+1}(t),y-x_i(t))$}
    &=\int_{x_{i+j}(t)}^{x_{i+j+1}(t)} -y_{i+1}(t)\wt'(\zeta_{i+1}(y,t))+y_i(t)\wt'(\zeta_{i}(y,t))dy\\
    &=(y_{i+1}(t)-y_i(t))\int_{x_{i+j}(t)}^{x_{i+j+1}(t)} -\wt'(\zeta_{i+1}(y,t))dy
    + y_i(t)\int_{x_{i+j}(t)}^{x_{i+j+1}(t)}\wt'(\zeta_{i}(y,t)) -\wt'(\zeta_{i+1}(y,t))dy
    \intertext{and since $\zeta_i(y,t)>\zeta_{i+1}(y,t)$ and $\wt$ is concave, we get}
    &\leq (y_{i+1}(t)-y_i(t))\int_{x_{i+j}(t)}^{x_{i+j+1}(t)} -\wt'(\zeta_{i+1}(y,t))dy.
\end{align*}
We are able to derive 
\begin{align*}
    \frac{d}{dt}\left(y_{i+1}(t)-y_i(t)\right)\geq &\, C(t)\left(y_{i+1}(t)-y_i(t)\right),
\end{align*}
where $C(t)$ collects all the estimates and again using Grönwall's inequality 
\begin{align*}
    \left(y_{i+1}(t)-y_i(t)\right)\geq &\left(y_{i+1}(0)-y_i(0)\right)\exp(\int_0^t C(s)ds)\geq 0.
\end{align*}
This proves that the monotonicity is kept. Furthermore, this allows to prove the lower bound on $y_i(t)$.
Therefore, we consider again the derivative \eqref{eq:dygeneral} and add zero such that we obtain
\begin{align*}
  \nonumber   \frac{d}{dt}y_i(t)=&\sum_{j=1}^{N-1-i} (\gamma_{i+1,j-1}(t)-\gamma_{i,j}(t))\left(v\left(\frac{1}{Ny_{i+j}(t)}\right)-v\left(\frac{1}{Ny_{i}(t)}\right)\right)\\
      \nonumber &+\left(\sum_{j=1}^{N-1-i} \gamma_{i+1,j-1}(t)-\sum_{j=0}^{N-1-i}\gamma_{i,j}(t)\right)\left(v\left(\frac{1}{Ny_{i}(t)}\right)-\bar v\right).
      \intertext{We use the monotonicity and the definition of the weights to get}
    \geq &-\int_{x_i(t)}^{x_{i+1}(t)}\wt(x_N-y)dy\left(v\left(\frac{1}{Ny_{i}(t)}\right)-\bar v\right)= y_i(t)\wt(\zeta^N_i(t)) \left(\bar v-v\left(\frac{1}{Ny_{i}(t)}\right)\right),
\end{align*}
where $\zeta^N_i(t)\in[x_N(t)-x_{i+1}(t),x_N(t)-x_i(t)]$.
Keeping in mind that $y_i(t)\leq \bar L$ holds and using again Grönwall's inequality yields
\begin{align*}
    y_i(t)\geq y_i(0)\exp\left(\int_0^t \wt(\zeta^N_i(t)) \left(\bar v-v\left(\frac{1}{Ny_{i}(t)}\right)\right)ds \right)\geq y_i(0).
\end{align*}
The case of monotone decreasing distances can be proven analogously.
Here, most of the inequalities switch their sign and one proves first $y_i(t)\geq \bar L$, then that under this assumption the monotonicity is kept and finally the upper bound.
\end{proof}
}
\end{appendix}

\bibliographystyle{plain}
\bibliography{mysources}
\end{document}